\documentclass[12pt,reqno]{amsart}

\numberwithin{equation}{section}

\usepackage{amssymb}
\usepackage{amsmath}
\usepackage{amscd}
\usepackage{amsthm}
\usepackage{amsfonts}
\usepackage{color}
\usepackage{ifpdf}
\ifpdf \usepackage[pdftex,pdfstartview=FitH,pdfpagemode=none,colorlinks,bookmarks,linkcolor=blue]{hyperref} \else  \usepackage[hypertex]{hyperref} \fi
\usepackage[nobysame,abbrev,alphabetic]{amsrefs}

\newtheorem{theorem}{Theorem}[section]

\newtheorem{lemma}[theorem]{Lemma}
\newtheorem{corollary}[theorem]{Corollary}
\newtheorem{definition}[theorem]{Definition}

\newtheorem{proposition}[theorem]{Proposition}

\newtheorem{remark}[theorem]{Remark}
\theoremstyle{definition}

\newcommand{\cL}{\mathcal{L}}

\newcommand{\cN}{\mathcal{N}}
\newcommand{\cO}{\mathcal{O}}

\newcommand{\cX}{\mathcal{X}}

\newcommand{\bC}{\mathbb{C}}

\newcommand{\bR}{\mathbb{R}}
\newcommand{\bZ}{\mathbb{Z}}
\newcommand{\bQ}{\mathbb{Q}}

\newcommand{\bN}{\mathbb{N}}

\newcommand{\bT}{\mathbb{T}}
\newcommand{\bS}{\mathbb{S}}

\newcommand{\bfe}{\mathbf{e}}
\newcommand{\bfm}{\mathbf{m}}
\newcommand{\bfn}{\mathbf{n}}

\newcommand{\GL}{\operatorname{GL}}

\newcommand{\rank}{\operatorname{rank}}
\newcommand{\Arg}{\operatorname{Arg}}
\newcommand{\Aut}{\operatorname{Aut}}

\newcommand{\imag}{\mathrm{i}}

\newcommand\diag[1]{\operatorname{diag}\left(#1\right)}
\newcommand{\onto}{\xymatrix{\ar@{>>}[r]&}}
\newcommand{\da}[4]{\xymatrix{#1 \ar@<.5ex>[r]^{#2} \ar@<-.5ex>[r]_{#3} & #4}}
\begin{document}

\title[Rigidity of commutative non-hyperbolic actions]{Rigidity of commutative non-hyperbolic actions by toral automorphisms}
\author{Zhiren Wang}
\address{\newline Department of Mathematics, Princeton University\newline Princeton, NJ 08544, USA\newline {\tt zhirenw@math.princeton.edu}}
\setcounter{page}{1}

\begin{abstract}
Berend obtained necessary and sufficient conditions on a $\bZ^r$-action $\alpha$ on a torus $\bT^d$ by toral automorphisms in order for every orbit be either finite or dense. One of these conditions is that on every common eigendirection of the $\bZ^r$-action there is an element $\mathbf n \in \bZ ^ r$ so that $\alpha^\mathbf n$ expands this direction. In this paper, we investigate what happens when this condition is removed; more generally, we consider a partial orbit $\left\{ \alpha^\mathbf n.x : \mathbf n \in \Omega \right\}$ where $\Omega$ is a set of elements which acts in an approximately isometric way on a given set of eigendirections. This analysis is used in an essential way in the work of the author with E. Lindenstrauss classifying topological self-joinings of maximal $\bZ^r$-actions on tori for $r\geq 3$.
\end{abstract}

\maketitle
{\small\tableofcontents}

\section{Introduction}

\subsection{Background}

In the landmark paper \cite{F67}, Furstenberg showed that any closed subset of $\bT=\bR/\bZ$ which is invariant under both $x\mapsto 2x$ and $x\mapsto 3x$ is either $\bT$ itself or a finite collection of rational points. This theorem was extended by Berend \cite{B83} to commutative semigroup actions on higher dimensional tori. A special case is the following:

\begin{theorem}(\cite{B83})\label{Berend} Let $\alpha$ be a faithful $\bZ^r$-action on $\bT^d=\bR^d/\bZ^d$ by toral automorphisms given by a group embedding $\alpha:\bfn\mapsto\alpha^\bfn$ of $\bZ^r$ into $\GL(d,\bZ)=\Aut(\bT^d)$ satisfying the following three conditions:\begin{enumerate}
\item $r \geq 2$;

\item $\exists\bfn\in\bZ^r$ such that $\alpha^\bfn$ is a totally irreducible toral automorphism;

\item for any common eigenvector $v\in\bC^d$ of $\alpha(\bZ^r)$, there exists $\bfn\in \bZ^r$ such that $|\alpha^\bfn.v|>|v|$.\end{enumerate}
Then $\forall x\in\bT^d$,  $\{\alpha^\bfn.x: \bfn\in\bZ^r\}$ is dense in $\bT^d$ unless $x$ is rational.
\end{theorem}

Here a toral automorphism $\phi\in\GL(d,\bZ)$ is said to be {\it irreducible} if there is no proper $\phi$-invariant subtorus of positive dimension in $\bT^d$, and it is {\it totally irreducible} if any power $\phi^k, k\neq 0$ is irreducible.

The aim of this paper is to investigate what happens when the hyperbolicity assumption (3) in Theorem \ref{Berend} fails.

Our result is going to cover two situations. The first one simply deals with $\bZ^r$-actions that don't satisfy assumption (3), i.e. there are one or several common eigenvectors whose norms are preserved by the group action.

In the second more general setup, we take a $\bZ^r$-action and fix several common eigenvectors $v_i\in\bC^d, i\in S$. Instead of studying the full $\bZ^r$-action. we only allow ourselves to apply those $\alpha^\bfn$'s satisfying $\frac{|\alpha^\bfn.v|}{|v|}\in(e^{-\epsilon},e^\epsilon), \forall i\in S$ to a point $x\in\bT^d$ and ask how the resulting orbit distributes in $\bT^d$. This case is more delicate as typically the elements we apply do not form a subgroup of $\bZ^r$.

In Theorem \ref{denseorcentral} we will give an analogue to Theorem \ref{Berend} in these situations with assumptions (1) and (2) properly reformulated.

Another question surrounding Berend's theorem is what happens when the action is no longer irreducible. In particular, it is interesting to investigate the diagonal action $\alpha_\triangle:\bZ^r\curvearrowright\bT^d\times\bT^d$ by $\alpha$, where $\alpha_\triangle^\bfn.(x,y)=(\alpha^\bfn.x,\alpha^\bfn.y)$, and ask what the orbit closures are. In a forthcoming joint paper \cite{LW10} with E.Lindenstrauss, we give a classification of orbit closures when $r\geq 3$ and the action $\alpha$ is Cartan. Theorem \ref{denseorcentral} is used as a key lemma in \cite{LW10}.

Before going further we give a more explicit form of the group actions described in Theorem \ref{Berend}.

\subsection{Number-theoretical description of the group action}

The irreducibility assumption on $\alpha$ has a number-theoretical interpretation.

Consider a number field $K$ of degree $d$. Suppose $K$ has $r_1$ real embeddings $\sigma_1,\cdots,\sigma_{r_1}$ and $r_2$ conjugate pairs of complex embeddings $(\sigma_{r_1+1},\sigma_{r_1+r_2+1})$, $(\sigma_{r_1+2},\sigma_{r_1+r_2+2})$, $\cdots$, $(\sigma_{r_1+r_2},\sigma_{r_1+2r_2})$, then $r_1+2r_2=d$ and the group of units $U_K$ has rank $r_1+r_2-1$. Recall $K\otimes_\bQ\bR\cong\bR^{r_1}\oplus\bC^{r_2}$ where the embedding of $K$ into $\bR^{r_1}\oplus\bC^{r_2}$ is given by
\begin{equation}\label{fieldemb}\sigma:\theta\mapsto\big(\sigma_1(\theta),\cdots,\sigma_{r_1}(\theta),\sigma_{r_1+1}(\theta),\cdots,\sigma_{r_1+r_2}(\theta)\big).\end{equation}

$K$ acts multiplicatively on $K\otimes_\bQ\bR$ by $\theta.(\mu\otimes x)=\theta\mu\otimes x, \forall \theta,\mu\in K, x\in\bR$; or equivalently, on $\bR^{r_1}\oplus\bC^{r_2}$ by
\begin{equation}\label{fieldmulti}\theta.(x_1,\cdots,x_{r_1+r_2})=\big(\sigma_1(\theta)x_1,\cdots,\sigma_{r_1+r_2}(\theta)x_{r_1+r_2}\big).\end{equation}
This multiplicative action is compatible with $\sigma$ in the sense that \begin{equation}\label{embmulti}\theta.\sigma(\mu)=\sigma(\theta\mu), \forall t,s\in K.\end{equation}

The following result makes the translation from the $\bT^d$ setting to a number-theoretical one. It is a special case of a more general fact, for which we refer to Schmidt \cite{S95}*{\S 7 \& \S 29} and Einsiedler-Lind \cite{EL04} (see also Einsiedler-Lindenstrauss \cite{EL03}*{Prop. 2.1}).

\begin{proposition}[cf.\ e.g.~\cite{W10}*{Prop. 2.13}]\label{Gfield}Let $\alpha:\bZ^r\hookrightarrow\GL(d,\bZ)$ be a $\bZ^r$-action on $\bT^d$ such that $\alpha(\bZ^r)$ contains an irreducible toral automorphism. Then there are \begin{itemize}
\item a number field $K$ of degree $d$ with $r_1$ real embeddings $\sigma_1\cdots,\sigma_{r_1}$ and $r_2$ conjugate pairs of complex embeddings $(\sigma_{r_1+1}, \sigma_{r_1+r_2+1})$, $(\sigma_{r_1+2}, \sigma_{r_1+r_2+2})$, $\cdots$, $(\sigma_{r_1+r_2}, \sigma_d)$ where $r_1+2r_2=d$;
\item a common eigenbasis in $\bC^d$ with respect to which $\forall \bfn\in\bZ^r$, $\alpha^\bfn$ can be diagonalized as $\diag{\zeta_1^\bfn,\zeta_2^\bfn,\cdots,\zeta_d^\bfn}$ where $\zeta_1^\bfn,\cdots,\zeta_{r_1}^\bfn\in\bR$; $\zeta_{r_1+1}^\bfn,\cdots,\zeta_d^\bfn\in\bC$ and $\zeta_{r_1+j}^\bfn=\overline{\zeta_{r_1+r_2+j}^\bfn},\forall j=1,\cdots,r_2$;
\item a group embedding $\zeta:\bfn\mapsto\zeta^\bfn$ of $\bZ^r$ into the group of units $U_K$;
\item a cocompact lattice $\Gamma$ in $\bR^{r_1}\oplus\bC^{r_2}$ that is contained in $\sigma(K)$, where $\sigma$ is given in (\ref{fieldemb});
\end{itemize}
such that:
\begin{itemize}
\item $\zeta_i^\bfn=\sigma_i(\zeta^\bfn),\forall i\in\{1,\cdots,d\},\forall \bfn\in \bZ^r$;
\item $\forall \bfn\in \bZ^r$, multiplication by $\zeta^\bfn$ on $\bR^{r_1}\oplus\bC^{r_2}$ as in (\ref{fieldmulti}) preserves $\Gamma$, hence induces an action on $(\bR^{r_1}\oplus\bC^{r_2})/\Gamma$ by \[\zeta^\bfn.(\tilde x\mathrm{\ mod\ }\Gamma)=(\zeta^\bfn.\tilde x\mathrm{\ mod\ }\Gamma), \forall\tilde x\in \bR^{r_1}\oplus\bC^{r_2};\]
\item the action $\alpha$ on $\bT^d$ is algebraically conjugate to the multiplicative action of $\bZ^r\curvearrowright(\bR^{r_1}\oplus\bC^{r_2})/\Gamma$ via $\zeta$, i.e.~there is a continuous group isomorphism $\psi:\bT^d\overset\sim\to(\bR^{r_1}\oplus\bC^{r_2})/\Gamma$ such that $\alpha^\bfn.x=\psi^{-1}(\zeta^\bfn.\psi(x))$, $\forall\bfn\in\bZ^r, \forall x\in\bT^d$.
\end{itemize}\end{proposition}

\begin{remark}\label{alphazetairr} For $\bfn\in\bZ^r$, $\alpha^\bfn$ is an irreducible toral automorphism if and only if $\zeta^\bfn$ doesn't belong to any non-trivial proper subfield of $K$. This is simply because $\zeta_i^\bfn$ is an eigenvalue of $\alpha^\bfn$. In particular, $\alpha^\bfn$ is totally irreducible if and only if $(\zeta^{\bfn})^k\notin L$ for all non-trivial proper subfield $L$ of $K$ and $k\in\bN$.\end{remark}

\subsection{Notations} \label{Xnota}Throughout the rest of paper we consider a $\bZ^r$-action $\alpha$ as in Proposition \ref{Gfield}. We will write $G$ for the acting group $\bZ^r$, and let $K$, $\Gamma$, $\psi$  and $\zeta$ be as in Proposition \ref{Gfield}.

Write $I=\{1,\cdots,r_1+r_2\}$.

For $1\leq i\leq r_1$, let $V_i$ be the $i$-th copy of $\bR$ in $\bR^{r_1}\oplus\bC^{r_2}$; while for $r_1< i\leq r_1+r_2$, let $V_i$ be the $(i-r_1)$-th copy of $\bC$. Then $\bR^{r_1}\oplus\bC^{r_2}$ rewrites $\bigoplus_{i\in I}V_i$. We write
\begin{equation}X=(\bR^{r_1}\oplus\bC^{r_2})/\Gamma=(\bigoplus_{i\in I}V_i)/\Gamma,\end{equation} and let $\pi$ be the canonical projection from  $\bR^{r_1}\oplus\bC^{r_2}$ to $X$. $X$ is a compact abelian group.

Moreover, as $\psi$ is a continuous group morphism from $\bT^d=\bR^d/\bZ^d$ to $X$. It can be easily seen that it lifts to a non-degenerate real linear map $\tilde\psi$ from $\bR^d$ to $(\bR^{r_1}\oplus\bC^{r_2})\cong\bR^d$ such that $\tilde\psi(\bZ^d)=\Gamma$ and $\alpha^n.\tilde x=\tilde\psi^{-1}(\zeta^\bfn.\tilde\psi(\tilde x)), \forall\bfn\in\bZ^r, \forall\tilde x\in\bR^d$.

The pushforward of the usual Euclidean distance by $\pi$ induces a Riemannian metric on $X$ and makes $X$ a locally Euclidean metric space. With this metric, the distance from a point $x$ to the origin $0$ in $X$ is given by
\begin{equation}\|x\|=\min_{\substack{\tilde x\in\bR^{r_1}\oplus\bC^{r_2}\\\pi(\tilde x)=x}}|\tilde x|,\end{equation} where $|\tilde x|=(\sum_{i\in I}|\tilde x_i|^2)^{\frac12}$, $\tilde x_i$ denoting the $V_i$ coordinate of $\tilde x$. The distance between $x,x'\in X$ is just $\|x-x'\|$.

For $x\in X$ and $D\subset G$, write $D.x=\{\zeta^\bfn.x: \bfn\in D\}$.

\begin{definition}\label{torsiontrans} For a real linear subspace $V$ of $\bR^{r_1}\oplus\bC^{r_2}$, a point $z\in X$ is a {\bf $V$-translated torsion point} if there are a torsion point $x_*\in X$ and a vector $v\in V$ such that $x=x_*+v$.

A {\bf $V$-disc centered at a torsion point} in $X$ is a set of the form $\{x_*+v: v\in V, |v|\leq R\}$ where $x_*$ is a given torsion point and $R$ is a given radius.\end{definition}

\begin{definition}\label{lya}For $i\in I$, the $i$-th {\bf Lyapunov functional} $\lambda_i:\bZ^r\mapsto\bR$ is given by \begin{equation}\lambda_i(\bfn)=\log|\zeta_i^\bfn|.\end{equation}
Moreover, define $\beta_i:\bZ^r\mapsto\bR/2\pi\bZ$ by \begin{equation}\beta_i(\bfn)=\arg\zeta_i^\bfn.\end{equation}\end{definition}

\begin{remark}\label{lyaextension}It is not hard to see $\lambda_i$ and $\beta_i$ are group morphisms. In particular, $\lambda_i$ extends uniquely to a linear map from $(\bR^r)^*$, which we still denote by $\lambda_i$.\end{remark}

In addition, though $\beta_i$ is defined for all $i\in I$, when $i\leq r_1$ it only takes value from $\{0,\pi\}$ modulo $2\pi\bZ$.

Write $\|\beta\|$ for the distance from $\beta\in\bR/2\pi\bZ$ to the trivial element $0$.

\subsection{Non-hyperbolic foliations}\label{ssCLS} The hyperbolicity condition (3) in Theorem \ref{Berend} is actually equivalent to assuming for all $i\in I$ that $\lambda_i$ is not the zero map from $(\bR^r)^*$ .

In this paper, we try to assume less hyperbolicity than \cite{B83} did. One way is to allow the zero map to appear as Lyapunov functionals. This gives an ``isometric subspace'' \begin{equation}\label{Berendcentral}V_{\mathrm{Isom}}=\bigoplus_{i\in I:\lambda_i\equiv 0}V_i\end{equation} that is invariant under the multiplicative action $\zeta$ and cannot be expanded or contracted by any element as $|\zeta_i^\bfn|=1, \forall\bfn$ whenever $\lambda_i\equiv0$.

Another way to lose hyperbolicity is to pose extra restrictions on which group elements from the action that one may apply. More precisely the question is: given $S\subset I$, for a generic $x\in \bT^d$ and $\epsilon>0$, is the partial orbit $\{\alpha^\bfn.x: \bfn\in\bZ^r\text{ s.t. }|\lambda_i|<\epsilon,\forall i\in S\}$ dense in $\bT^d$?

In fact we will consider even smaller truncations of $\bZ^r$ as in the next definition.

\begin{definition}\label{groupslice}If $S\subset I$, $\epsilon>0$ and $\sigma+H$ is a coset of some subgroup $H \leq G$,  the {\bf $\epsilon$-slice of $\sigma+H$ with respect to $S$} is
\begin{equation}H^\sigma_{\epsilon,S}=\{\bfn\in\sigma+H:|\lambda_i(\bfn)|<\epsilon,\forall i \in S\text{ and }\|\beta_j(\bfn)\|<\epsilon,\forall j\in I\}.\end{equation}When $\sigma+H$ is $H$ itself, simply write $H_{\epsilon,S}$ for $H^\sigma_{\epsilon,S}$.

$\sigma+H$ is said to be {\bf compatible} with $S$ if  $H^\sigma_{\epsilon,S}\neq\emptyset$, for all $\epsilon>0$. \end{definition}

Remark $H_{\epsilon,S}$ contains the identity hence $H$ itself is always compatible. In particular, one wishes to understand the action of the non-empty set \begin{equation}\label{GepsilonS}G_{\epsilon,S}=\{\bfn\in\bZ^r:|\lambda_i(\bfn)|<\epsilon,\forall i \in S\text{ and }\|\beta_j(\bfn)\|<\epsilon,\forall j\in I\}.\end{equation} We emphasize that $G_{\epsilon,S}$, and $H_{\epsilon,S}$ in general, usually don't form a subgroup of $G=\bZ^r$ .

\subsection{The main result}

The following theorem is the main result of this paper.

\begin{theorem}\label{denseorcentral}
Let $\alpha$ be a $\bZ^r$-action on $\bT^d$ by toral automorphisms. Assume that $\alpha^\bfn$ is irreducible for at least one $\bfn \in\bZ^r$, hence the notations from \S\ref{Xnota} are applicable. For any subset $S \subset I$, let $L_S \subset(\bR^r)^*$ be the $\bR$-span of the set $\{ \lambda_j: j \in S \}$ (or $ L_S=\{ 0 \}$ if $S=\emptyset$).  Let $\langle S \rangle \supset S$ be the set of all $i \in I$ such that $\lambda_i \in L_S$. Denote $V_{\langle S \rangle}=\bigoplus_{i \in \langle S \rangle}V_i$. Assume further that
\begin{enumerate}
\item $\dim{L_S}\leq r-2$;
\item $\forall \epsilon>0$, $\exists\bfn \in\bZ^r$ such that $\alpha^\bfn$ is a totally irreducible toral automorphism and $|\lambda_i(\bfn)|<\epsilon,\forall i \in S$.
\end{enumerate}
Then for any $\epsilon>0$, a point $x \in X$ satisfies that
\begin{equation}\label{Xpartialorb}\{ \zeta^\bfn.x: \bfn \in\bZ^r\text{ s.t. }|\lambda_i(\bfn)|<\epsilon,\forall i \in S \}\quad\text{is dense in $X$}\end{equation} if and only if $x$ is not a $V_{ \langle S \rangle}$-translated torsion point.
Equivalently, $\forall \epsilon>0$, $y \in \bT^d$ \begin{equation}\label{Tpartialorb}\{ \alpha^\bfn.y: \bfn \in\bZ^r\text{ s.t. }|\lambda_i(\bfn)|<\epsilon,\forall i \in S \}\quad\text{is dense in $\bT^d$}\end{equation}  if and only if $y$ can not be written as $y_*+v$, where $y_*\in\bT^d$ is a rational point and $v\in\tilde\psi^{-1}(V_{ \langle S \rangle})$.\end{theorem}

Recall $\tilde\psi$ is the lift of $\psi$ described in \S\ref{Xnota}.

The restricted dynamics studied here are special cases of non-expansive subdynamics of $\bZ^d$-actions defined and investigated by Boyle and Lind \cite{BY97}.

\begin{remark} In the special case $S=\emptyset$, the set (\ref{Xpartialorb}) is just the full orbit $G.x$  for all $\epsilon>0$ and the theorem studies exactly the same situation as in Theorem \ref{Berend} except that assumption (3) is removed. In this case $V_{\langle S\rangle}$ is exactly the isometric subspaces $V_{\mathrm{Isom}}$ defined in (\ref{Berendcentral}). Moreover the subspace $\tilde\psi^{-1}(V_{\langle S\rangle})=\tilde\psi^{-1}(V_{\mathrm{Isom}})$ can be defined as $\{v\in\bR^d:\exists C>0, C^{-1}|v|\leq |\alpha^\bfn.v|\leq C|v|\}$.\end{remark}

Instead of Theorem \ref{denseorcentral}, we will prove the following slightly stronger statement:

\begin{theorem}\label{denseorcentral'}Suppose $S\subset I$ satisfies both conditions (1) and (2) in Theorem \ref{denseorcentral}. Then for any point $x\in X$ we have the following dichotomy:\begin{enumerate}
\item[(i)] If $x$ is not a $V_{\langle S\rangle}$-translated torsion point, then $\bigcap_{\epsilon>0}\overline{G_{\epsilon,S}.x}=X$.
\item[(ii)] Otherwise $\bigcap_{\epsilon>0}\overline{G_{\epsilon,S}.x}$ is a finite set of $V_{\langle S\rangle}$-translated torsion points.\end{enumerate}
\end{theorem}

The spirit of the argument is that though usually not a subgroup, the subset $G_{\epsilon,S}$ looks like an abelian group of rank $r-\dim{L_S}$. As this rank is at least $2$ by assumption, techniques from \cite{B83} still apply.

\subsection{Necessity of total irreducibility}

\begin{remark}\label{totnecc}Assumption (2) in the theorem is necessary. In fact one can construct a $\bZ^r$-action $\alpha$ by toral automorphisms together with a set $S\subset I$ that meet all the other requirements in Theorem \ref{denseorcentral}, and a number $\epsilon_0>0$ such that: \begin{itemize}
\item[(2')] $\exists\bfn$ such that $|\lambda_i(\bfn)|<\epsilon_0,\forall i\in S$ and $\alpha^\bfn$ is totally irreducible;
\item[($\star$)] $\exists x\in X$ such that $x$ is not a $V_{ \langle S\rangle}$-translated torsion point but the subset \begin{equation}\label{counterexset}\{ \zeta^\bfn.x: \bfn \in\bZ^r\text{ s.t. }|\lambda_i(\bfn)|<\epsilon_0,\forall i \in S \}\end{equation} is not dense in $X$ either.\end{itemize}\end{remark}

A counterexample can be constructed in the follows way. The idea is that if the irreducibility assumption is not strong enough, then when one gets rid of hyperbolicity, irreducibility may be lost as well.

Let \begin{equation}\theta=\sqrt{\frac{\sqrt 6+\sqrt 2}2-1}-1\end{equation} and $K$ be the octic number field $\bQ(\theta)$. It is easy to check that $\deg K=8$, and $K$ has two real embeddings $\sigma_1,\sigma_2$, as well as three pairs of conjugate imaginary embeddings $(\sigma_i,\sigma_{i+5})$, where $i=2,3,4$. Moreover, $\sigma_1(\theta)=\sqrt{\frac{\sqrt 6+\sqrt 2}2-1}-1$ and $\sigma_2(\theta)=-\sqrt{\frac{\sqrt 6+\sqrt 2}2-1}-1$.   $\theta$ is an algebraic unit as it solves the polynomial $\theta^8+8\theta^7+32\theta^6+80\theta^5+132\theta^4+144\theta^3+96\theta^2+32\theta+1$. It is not hard to check $\frac{\sigma_i(\theta)}{\sigma_j(\theta)}$ is not a root of unity for all $i\neq j$ and it follows that for any $k \in N$ and any proper subfield $L\subset K$, $\theta^k\notin L$.

$K$ contains a totally real quartic subfield $F=\bQ(\sqrt 2,\sqrt 3)$. Remark $\sigma_1|_F=\sigma_2|_F$. By Dirichlet's Unit Theorem, $\rank(U_K)=4$ and $\rank(U_F)=3$. $\sqrt 3-\sqrt 2$, $\sqrt 2-1$ and $\frac{\sqrt6+\sqrt2}2$ form a set of fundamental units of $F$; furthermore, together with $\theta$, they give a system of fundamental units of $K$ as $\theta$ is not a root of unity. Define a group embedding $\zeta:\bZ^4\mapsto U_K$ by setting \begin{equation}\zeta^{\bfe_1}=\theta,\ \zeta^{\bfe_2}=\sqrt 3-\sqrt 2,\ \zeta^{\bfe_3}=\sqrt 2-1,\ \zeta^{\bfe_4}=\frac{\sqrt6+\sqrt2}2;\end{equation} where $\bfe_1,\cdots,\bfe_4$ is a basis of $\bZ^4$. As explained earlier, $\zeta$ defines naturally a multiplicative $\bZ^4$-action on $X=(\bR^2\oplus\bC^3)/\sigma(\cO_K)$ where $\cO_K$ is the ring of integers in $K$ and $\sigma$ is the canonical embedding of $K$ into $\bR^2\oplus\bC^3$ given by (\ref{fieldemb}). By choosing an arbitrary $\bZ$-basis of $\cO_K$, we can easily identify $X$ with $\bT^8$ and get a conjugate $\bZ^r$-action $\alpha$ by toral automorphisms as in Proposition \ref{Gfield}.

Let $S=\{1,2\}$One may explicitly compute the Lyapunov exponent $\lambda_i(\bfn)$ for each $i$ from $I=\{1,\cdots,5\}$ and verify that $\lambda_1$ and $\lambda_2$ are not proportional, and none of $\lambda_3$, $\lambda_4$, $\lambda_5$ is inside the linear span $L_S$ of $\lambda_1$, $\lambda_2$. Hence $\langle S\rangle=S$ and $\dim L_S=2$. So the condition (1) from Theorem \ref{denseorcentral} is satisfied as $r=4$ in this case.

Note $|\sigma_1(\theta)|\neq |\sigma_2(\theta)|$, hence \begin{equation}\lambda_1(\bfe_1)-\lambda_2(\bfe_1)=\log|\sigma_1(\theta)|-\log|\sigma_2(\theta)|\neq 0.\end{equation} Take $\epsilon_0>\max\big(\log|\sigma_1(\theta)|,\log|\sigma_2(\theta)|\big)$. Then $\bfe_1$ satisfies $|\lambda_i(\bfe_1)|<\epsilon_0$ for both $i=1,2$. By Remark \ref{alphazetairr}, the fact that $(\zeta^{\bfe_1})^k=\theta^k$ does not belong to any proper subfield of $K$ implies $\alpha^{\bfe_1}$ is a totally irreducible automorphism on $\bT^8$. Therefore we verified the condition (2') in Remark \ref{totnecc}.

Last, we are going to establish the claim ($\star$). Consider all $\bfn\in\bZ^4$ such that \begin{equation}\label{lyaboundcounter}|\lambda_i(\bfn)|<\epsilon_0,i=1,2,\end{equation} write $\bfn=a\bfe_1+\bfn'$ where $a\in\bZ$ and $\bfn'\in\bZ\bfe_2\oplus\bZ\bfe_3\oplus\bZ\bfe_4$. It follows that $\zeta^{\bfn'}\in U_F$. Since $\sigma_1$ and $\sigma_2$ coincide on $F$, we see  $\lambda_1(\bfn')=\lambda_2(\bfn')$. Thus $\lambda_1(\bfn)-\lambda_2(\bfn)=\lambda_1(a\bfe_1)-\lambda_2(a\bfe_1)=a\big(\lambda_1(\bfe_1)-\lambda_2(\bfe_1)\big)$ and it follows from (\ref{lyaboundcounter}) that $|a|$ is bounded by $\left|\frac{2\epsilon_0}{\lambda_1(\bfe_1)-\lambda_2(\bfe_1)}\right|$. Hence $a$ can take only finitely many values.

Remark that under the natural projection $\pi$ from $\bR^2\oplus\bC^3$ to $X$, the $4$-dimensional real subspace $\overline{\sigma(F)}\subset \bR^2\oplus\bC^3$ has a closed image $\pi\big(\overline{\sigma(F)}\big)=\overline{\sigma(F)}/\sigma(\cO_F)$, which is a $4$-dimensional subtorus in the $8$-dimensional twisted torus $X$, which we denote by $Y$. As $U_F$ preserves $F$ under multiplication, the restriction of the action by $U_K$ on $X$ to $U_F$ stabilizes $Y$. In particular $\zeta^{\bfn'}.Y=Y$. Hence $\zeta^\bfn.Y=\zeta^{a\bfe_1}\zeta^{\bfn'}.Y=\zeta^{a\bfe_1}.Y$. Note for each fixed $a$, $\zeta^{a\bfe_1}.Y$ is a $4$-dimensional subtorus of $X$. Because under the assumption (\ref{lyaboundcounter}), there are only finitely many choices of $a$, for any $x\in Y$ the set (\ref{counterexset}) is contained in a finite union of $4$-dimensional subtori and hence cannot be dense in $X$. But it is not hard to show that since $V_{\langle S\rangle}=V_S$ is a proper $2$-dimensional subspace of $\bR^2\oplus\bC^3$ in this case, not every $x\in Y$ can be represented as a $V_{\langle S\rangle}$-translated torsion point. This produces a counterexample as described in Remark \ref{totnecc}.

\section{$(H,S)$-invariant sets}

In order to show Theorem \ref{denseorcentral'} it is necessary to have some extra formulations, mainly to overcome the obstacle that a typical $G_{\epsilon,S}$ doesn't form a group.

\begin{definition}For $S\subset I$ and a subgroup $H \leq G$, a closed subset $A$ of $X$ is said to be $(H,S)$-invariant if for all $x\in A$, $\bigcap_{\epsilon>0}\overline{H_{\epsilon,S}.x}\subset A$.\end{definition}

\subsection{Relation to Katok and Spatzier's suspension construction}\label{suspension} The $(H,S)$-invariant sets in this paper are introduced in the spirit of subdynamics along foliations defined by Boyle-Lind \cite{BY97}. They are also inspired by, and closely related to, the suspension construction of Katok and Spatzier.

Given a $\bZ^r$-action $\rho$ on a space $N$, sometimes it is desirable to pass to an $\bR^r$-action. For this purpose, in \cite{KS96} Katok and Spatzier introduced a suspension space $\cN=(\bR^r\times N)/\bZ^r$, where the quotient is defined by the following $\bZ^r$-action on $\bR^r\times N$: \begin{equation}\bfn.(\eta,x)=(\eta-\bfn,\rho(\bfn).x),\ \forall\bfn\in\bZ^r,\forall (\eta,x)\in \bR^r\times N.\end{equation} For all $\eta\in\bR^r$ and $x\in N$, denote by $\overline{(\eta,x)}$ the equivalence class from $\cN$ that contains $(\eta,x)$.

Note that the additive action $\bR^r\curvearrowright\bR^r\times N$ given by $\eta'.(\eta,x)=(\eta+\eta',x)$ commute with the quotient structure and induces an $\bR^r$-action on $\cN$, which is denoted by $\tilde\rho$.

The space $\cN$ has a natural fiber structure over $\bT^r=\bR^r/\bZ^r$, where the fiber above the equivalent class $\bar\eta=\eta+\bZ^r$ is $\big\{\overline{(\eta,x)}:x\in N\big\}$ and is homeomorphic to $N$. For $\eta\in\bR^r$ and $x\in X$, suppose $\cO$ is the orbit of $(0,x)$ under the $\bR^r$-action $\tilde\rho$, then its intersection with the fiber above $\bar\eta$ writes $\{\overline{(\eta,y)}:\exists\bfn\in\bZ^r,y=\rho(\bfn).x\}$. In particular for any $\tilde\rho$-orbit $\cO\subset\cN$, the intersection of $\cO$ with any fiber, which we identify with $N$, is an orbit of the $\bZ^r$-action $\rho$. Moreover, since $\cO$ is $\tilde\rho$-invariant, its intersections with different fibers are homeomorphic to each other. Hence by classifying $\tilde\rho$-orbits in $\cN$, one also classifies $\rho$-orbits in $N$.

In our case, let $N=X$ and $\rho$ be the multiplicative $\bZ^r$-action $\zeta$, and construct the suspension system $(\cX,\tilde\zeta)$ in the above way. It follows from Berend's Theorem \ref{Berend} and previous discussion that the $\tilde\zeta$-orbit of any point $\overline{(\eta,x)}\in\cX$ in $\cX$ unless $x$ is not a torsion point in $X$.

\begin{definition}\label{planenonhyp}For any subset $S\subset I$, define a hyperplane \begin{equation}P_S=\{\eta\in\bR^r:\lambda_i(\eta)=0,\forall i\in S\}.\end{equation}\end{definition}
Notice $P_S=L_S^\bot$ and $L_S=P_S^\bot$ where $L_S\subset(\bR^r)^*$ is defined in the statement of Theorem \ref{denseorcentral}. By assumption (1) in Theorem \ref{denseorcentral}, \begin{equation}\dim P_S=r-\dim L_S\geq 2.\end{equation}

In Theorem \ref{denseorcentral} we study the partial orbit of a point $x\in X$ under the action of all elements from $\bZ^r$ that are very close to the hyperplane $P_S\subset\bR^r$. This is linked to the partial orbit $\cO_{\tilde\zeta}\big(P_S, \overline{(0,x)}\big)$ of $\overline{(0,x)}\in\cX$ under the restriction of the action $\tilde\zeta$ to $P_S\subset\bR^r$. Actually, it is possible to show that Theorem \ref{denseorcentral} is equivalent the claim that $\cO_{\tilde\zeta}\big(P_S, \overline{(0,x)}\big)$ is dense in $\cX$ unless $x$ is a $V_{\langle S\rangle}$-translated torsion point.\newline

Theorem \ref{denseorcentral'} is more precise than Theorem \ref{denseorcentral} and corresponds to a slightly more complicated kind of suspension systems that takes complex conjugates into account.

Let $\bZ^r$ act on the space $\bR^r\times(\bR/2\pi\bZ)^{r_1+r_2}\times X$ by
\begin{equation}\label{suspensionprod}\bfn.\big(\eta,\omega,x\big)=\Big(\eta-\bfn,\omega-\big(\beta_j(\bfn)\big)_{j\in I},\zeta^\bfn.x\Big),\end{equation} which induces a compact quotient $(\bR^r\times(\bR/2\pi\bZ)^{r_1+r_2}\times X)/\bZ^r$ that we denote by $\cX^\circ$. For $(\eta,\omega,x)\in \bR^r\times(\bR/2\pi\bZ)^{r_1+r_2}\times X$, denote by $\overline{(\eta,\omega,x)}$ its projection in $\cX^\circ$.

Remark that
\begin{equation}\label{suslattice}\Omega=\Big\{\Big(\bfn,\big(\tilde\beta_j(\bfn)\big)_{j\in I}\Big):\bfn\in\bZ^r\Big\}\end{equation} is a discrete cocompact subgroup in  $\bR^r\times(\bR/2\pi\bZ)^{r_1+r_2}$. Note the quotient $(\bR^r\times(\bR/2\pi\bZ)^{r_1+r_2})/\Omega$ is isomorphic to $\bT^{r+r_1+r_2}$. For $(\eta,\omega)\in \bR^r\times(\bR/2\pi\bZ)^{r_1+r_2}$, let $\overline{(\eta,\omega)}$ be its projection in $(\bR^r\times(\bR/2\pi\bZ)^{r_1+r_2})/\Omega$.

Then $\cX^\circ$ fibers naturally over $(\bR^r\times(\bR/2\pi\bZ)^{r_1+r_2})/\Omega$ and each fiber being a copy of $X$.

The group $\bR^r\times(\bR/2\pi\bZ)^{r_1+r_2}$ acts naturally on $\bR^r\times(\bR/2\pi\bZ)^{r_1+r_2}\times X$ by translation on the first two factors. Remark this action commutes with the $\bZ^r$-action \ref{suspensionprod}, hence passes to a $\bR^r\times(\bR/2\pi\bZ)^{r_1+r_2}$-action on the quotient space $\cX^\circ$, which we still denote by $\tilde\zeta$ without causing ambiguity.

The group $H_{\epsilon,S}$ consists of elements $\bfn\in H$ such that $\big(\bfn,(\beta_j(\bfn))_{j\in I}\big)$ is sufficiently close to the subgroup $P_S\times\{0\}$ of $\bR^r\times(\bR/2\pi\bZ)^{r_1+r_2}$. Then the set $\bigcap_{\epsilon}\overline{H_{\epsilon,S}.x}$ is closely related to the restriction of the action $\tilde\zeta$ to $P_S\times\{0\}$. In fact, if $\cO_{\tilde \zeta}\big(P_S\times\{0\},\overline{(0,0,x)}\big)$ denotes the orbit of $\overline{(0,0,x)}\in\cX^\circ$ under the restriction of $\tilde\zeta$ to the subgroup $P_S\times\{0\}$, then one can show that the intersection between its closure and the fiber above $\overline{(\eta,\omega)}$ is exactly \begin{equation}\big\{\overline{(\eta,\omega,y)}:y\in \bigcap_{\epsilon>0}\overline{G_{\epsilon,S}.x}\big\}.\end{equation} Thus another way to formulate Theorem \ref{denseorcentral'} is that:\newline

{\it $\cO_{\tilde \zeta}\big(P_S\times\{0\},\overline{(0,0,x)}\big)$ is dense in $\cX^\circ$ unless $x$ is a $V_{\langle S\rangle}$-translated torsion point.}\newline

It should be emphasized that though the results in this paper are proved using $(H,S)$-invariant sets, they can also obtained by studying the suspension systems described above.

\subsection{Basic properties of invariant sets}\label{basicinv}

Despite the fact that the $(H,S)$-invariant sets are defined using the $H_{\epsilon,S}$'s which are not groups in general, to some extent they have similar properties to invariant sets under group actions.

\begin{lemma}\label{cosetcompo}(i) Suppose $H \leq G$ and two cosets $\sigma+H$, $\tau+H$ are both compatible with $S$, then so is $\sigma+\tau+H$;

(ii) $\forall x\in X$, $\forall\epsilon>0$, $H^\tau_{\delta,S}.\overline{H^\sigma_{\epsilon,S}.x}\subset \overline{H^{\sigma+\tau}_{\epsilon+\delta,S}.x}$.
\end{lemma}

\begin{proof}(i) It suffices to show for any $\epsilon,\delta>0$, $H^{\sigma+\tau}_{\epsilon+\delta,S}$ is non-empty. By assumption both $H^\sigma_{\epsilon,S}$ and $H^\tau_{\delta,S}$ are non-empty, from which we respectively take elements $\bfm$ and $\bfn$. Then $\bfm+\bfn\in\sigma+\tau+H$ as $\bfm\in\sigma+H$, $\bfn\in\tau+H$. Furthermore by Definition \ref{groupslice},
\begin{equation}|\lambda_i(\bfm+\bfn)|\leq |\lambda_i(\bfm)|+|\lambda_i(\bfn)|<\epsilon+\delta, \forall i\in S;\end{equation} and \begin{equation}|\beta_j(\bfm+\bfn)|\leq |\beta_j(\bfm)|+|\beta_j(\bfn)|<\epsilon+\delta, \forall j\in I.\end{equation}
Hence $\bfm+\bfn\in H^{\sigma+\tau}_{\epsilon+\delta,S}$.

(ii) It is enough to prove for all $\bfn\in H^\tau_{\delta,S}$ that $\zeta^\bfn. \overline{H^\sigma_{\epsilon,S}.x}\subset \overline{H^{\sigma+\tau}_{\epsilon+\delta,S}.x}$. Since $\zeta^\bfn$ is a continuous map, it suffices to show $\zeta^\bfn. (H^\sigma_{\epsilon,S}.x)\subset H^{\sigma+\tau}_{\epsilon+\delta,S}.x$. However by the proof of part (i), for any $\bfm\in H^\sigma_{\epsilon,S}$, $\bfm+\bfn\in H^{\sigma+\tau}_{\epsilon+\delta,S}$; which completes the proof.
\end{proof}

\begin{corollary}\label{cosetsubgp}Suppose $H \leq \tilde H \leq G$ and $H$ is of finite index in $\tilde H$, then the cosets $\{\sigma+H:\sigma\in\tilde H\text{ s.t. }\sigma+H\text{ is compatible with } S\}$ form a subgroup of $\tilde H/H$.\end{corollary}
\begin{proof}By Lemma, the family of such cosets is stable under addition and contains the trivial element in the finite additive group $\tilde H/H$, hence is a subgroup.\end{proof}

\begin{corollary}\label{cosetcomposlice}(i) Suppose two coset $\sigma+H$, $\tau+H$ are both compatible with $ S$, then for all $y\in\bigcap_{\epsilon>0}\overline{H^\sigma_{\epsilon,S}.x}$,  $\bigcap_{\epsilon>0}\overline{H^\tau_{\epsilon,S}.y}\subset \bigcap_{\epsilon>0}\overline{H^{\sigma+\tau}_{\epsilon,S}.x}$;

(ii) If $\sigma+H$ is compatible with $S$, then $\forall x\in X$, the closed set $\bigcap_{\epsilon>0}\overline{H^\sigma_{\epsilon,S}.x}$ is non-empty and $(H,S)$-invariant.\end{corollary}

In particular, $\bigcap_{\epsilon>0}\overline{H_{\epsilon,S}.x}$ always contains $x$ and is $(H,S)$-invariant, which is analogous to the fact that orbit closures are invariant in the setting of group actions.

\begin{proof}(i) By Lemma \ref{cosetcompo}.(ii),
\begin{equation}\bigcap_{\epsilon>0}\overline{H^\tau_{\epsilon,S}.y}\subset\bigcap_{\epsilon>0}\overline{H^\tau_{\epsilon,S}.\overline{H^\sigma_{\epsilon,S}.x}}\subset\bigcap_{\epsilon>0}\overline{H^{\sigma+\tau}_{2\epsilon,S}.x}=\bigcap_{\epsilon>0}\overline{H^{\sigma+\tau}_{\epsilon,S}.x}.\end{equation}

(ii) Invariance follows from part (i) by taking $\tau=0$. Since $\sigma+H$ is compatible with $S$, $\overline{H^\sigma_{\epsilon,S}.x}$ is non-empty for all $\epsilon$, by compactness of $X$, the limit $\bigcap_{\epsilon>0}\overline{H^\sigma_{\epsilon,S}.x}$ is non-empty.
\end{proof}

\begin{remark}\label{describeinv} Clearly $\bigcap_{\epsilon>0}\overline{H_{\epsilon,S}.x}$ is the smallest $(H,S)$-invariant closed set containing $x$. Moreover it is not hard to see that $x'\subset\bigcap_{\epsilon>0}\overline{H^\sigma_{\epsilon,S}.x}$ if and only if there is a sequence $\{\bfn_k\}_{k=1}^\infty$ such that $\bfn_k\in H^\sigma_{\epsilon_k,S}$ where $\epsilon_k\rightarrow 0$, such that $\lim_{k\rightarrow\infty}\zeta^{\bfn_k}.x=x'$.\end{remark}

The next property is that the family of invariant sets is stable under addition and subtraction.

\begin{lemma}\label{invpm}Suppose two closed sets $A$ and $B$ are both $(H,S)$-invariant, then so are $A+B=\{x+y:x\in A,y\in B\}$ and $A-B=\{x-y:x\in A,y\in B\}$.\end{lemma}
\begin{proof} Suppose $z=x+y\in A+B$ where $x\in A$, $y\in B$. Take $z'\in\bigcap_{\epsilon>0}\overline{H_{\epsilon,S}.z}$, then by Remark \ref{describeinv}, there is a sequence $\epsilon_k\rightarrow 0$ and elements $\bfn_k\in H_{\epsilon_k,S}$ such that $\lim_{k\rightarrow\infty}\zeta^{\bfn_k}.z=z'$. As $X$ is compact, by passing to a subsequence we may assume $\lim_{k\rightarrow\infty}\zeta^{\bfn_k}.x=x'$ and $\lim_{k\rightarrow\infty}\zeta^{\bfn_k}.y=y'$, which belong respectively to $\bigcap_{\epsilon>0}\overline{H_{\epsilon,S}.x}$ and $\bigcap_{\epsilon>0}\overline{H_{\epsilon,S}.y}$. Hence $x'\in A$, $y'\in B$ and $z'=x'+y'\in A+B$. So $A+B$ is invariant. The proof for $A-B$ goes the same.\end{proof}

The next lemma allows us to talk about ``finitely generated'' invariant sets.

\begin{lemma}\label{fginvariant}For a finite set of points $x_1,\cdots,x_N\in X$ the  closed set \begin{equation}\label{fginvariant00}\bigcap_{\epsilon>0}\overline{H_{\epsilon,S}.x_1}+\cdots+\bigcap_{\epsilon>0}\overline{H_{\epsilon,S}.x_N}\end{equation} is $(H,S)$-invariant. Moreover,

(i) The set (\ref{fginvariant00}) is equal to $\bigcap_{\epsilon>0}(\overline{H_{\epsilon,S}.x_1}+\cdots+\overline{H_{\epsilon,S}.x_N})$;

(ii) $H_{\delta,S}.(\overline{H_{\epsilon,S}.x_1}+\cdots+\overline{H_{\epsilon,S}.x_N})\subset \overline{H_{\epsilon+\delta,S}.x_1}+\cdots+\overline{H_{\epsilon+\delta,S}.x_N}$ for all $\epsilon>0$.\end{lemma}

\begin{proof}  The invariance follows from the previous lemma.

It is clear that $(\ref{fginvariant00})\subset\bigcap_{\epsilon>0}(\overline{H_{\epsilon,S}.x_1}+\cdots+\overline{H_{\epsilon,S}.x_N})$; while the proof of the other direction of part (i) is basically the same as that of the previous lemma.

Part (ii) is an immediate corollary to Lemma \ref{cosetcompo}.(ii).\end{proof}

\subsection{Minimal invariant sets}

It is easy to see that $(H,S)$-invariant closed sets satisfy the descending chain condition: if $A_1\supset A_2\supset\cdots$ is a sequence of decreasing non-empty $(H,S)$-invariant closed sets, then the limit set $A=\bigcap_{n=1}^\infty A_n$ is also non-empty and $(H,S)$-invariant. Therefore it follows directly from Zorn's Lemma that it makes sense to talk about minimal invariant closed sets:

\begin{lemma}\label{minZorn}For a subgroup $H \leq G$ and $S\subset I$, any non-empty $(H,S)$-invariant closed set $A$ contains a {\bf minimal} $(H,S)$-invariant closed set $M$, i.e. $M$ is non-empty and $(H,S)$-invariant, and has no non-empty proper closed subset which is also $(H,S)$-invariant.\end{lemma}

\begin{proof}Apply Zorn's Lemma.\end{proof}

For a minimal set $M$, any point $x\in M$ ``generates'' $M$.

\begin{lemma}\label{mingene} Let $M$ be a minimal $(H,S)$-invariant closed set, then for all $x\in M$, $\bigcap_{\epsilon>0}\overline{H_{\epsilon,S}.x}=M$.
\end{lemma}
\begin{proof}By definition of $(H,S)$-invariant closed sets, $\bigcap_{\epsilon>0}\overline{H_{\epsilon,S}.x}\subset M$. By the remark following Corollary \ref{cosetcomposlice}, $\bigcap_{\epsilon>0}\overline{H_{\epsilon,S}.x}$ is non-empty and $(H,S)$-invariant. So $M$ cannot be minimal unless $\bigcap_{\epsilon>0}\overline{H_{\epsilon,S}.x}=M$.\end{proof}

Let $H \leq \tilde H \leq G$ be two subgroups such that $|\tilde H/H|<\infty$. An $(\tilde H,S)$-invariant closed set is necessarily $(H,S)$-invariant; however it is not obvious whether a minimal $(\tilde H,S)$-invariant closed set should also be minimal in $(H,S)$-sense. The following result gives a relationship between these two classes of minimal invariant sets.

\begin{proposition}\label{minunion}Let $H \leq \tilde H \leq G$ with $|\tilde H/H|<\infty$ and $M\subset X$ be a minimal $(\tilde H,S)$-invariant closed set. Denote by $N$ the number of cosets from $\tilde H/H$ that are compatible with $S$. Then there exist $N$ minimal $(H,S)$-invariant closed subsets $M_1,\cdots,M_N\subset M$ whose union is $M$.\end{proposition}

\begin{proof}Write the $N$ compatible cosets as $\sigma_1+H,\cdots,\sigma_N+H$ where $\sigma_1=\mathbf 0$.\\

{\noindent\bf Step 1}. First of all, we claim that \begin{equation}\label{minunion1}\bigcup_{n=1}^N\bigcap_{\epsilon>0}\overline{H^{\sigma_n}_{\epsilon,S}.x}=M,\ \forall x\in M\end{equation}

Actually for any $x\in M$, by Lemma \ref{mingene} $\bigcap_{\epsilon>0}\overline{\tilde H_{\epsilon,S}.x}=M$ . Since $\forall n, \bigcap_{\epsilon>0}\overline{H^{\sigma_n}_{\epsilon,S}.x}\subset \bigcap_{\epsilon>0}\overline{\tilde H_{\epsilon,S}.x}$, the left-hand side in (\ref{minunion1}) is contained in $M$. So it suffices to show $\bigcup_{n=1}^N\bigcap_{\epsilon>0}\overline{H^{\sigma_n}_{\epsilon,S}.x}\supset M$.

As $\tilde H_{\epsilon,S}=\bigcup_{n=1}^N H^{\sigma_n}_{\epsilon,S}$, $\tilde H_{\epsilon,S}.x$ is equal to the finite union $\bigcup_{n=1}^NH^{\sigma_n}_{\epsilon,S}$, hence by finiteness $\overline{\tilde H_{\epsilon,S}.x}=\bigcup_{n=1}^N\overline{H^{\sigma_n}_{\epsilon,S}.x}$.

So for all $z\in M=\bigcap_{\epsilon>0}\overline{\tilde H_{\epsilon,S}.x}=\bigcap_{\epsilon>0}\bigcup_{n=1}^N\overline{H^{\sigma_n}_{\epsilon,S}.x}$ and all $k\in\bN$, there exists $n_k\in\{1,\cdots,N\}$ such that $z\in \overline{H^{\sigma_{n_k}}_{k^{-1},S}.x}$. Thus there is a subsequence $k_l\rightarrow\infty$ such that all the $n_{k_l}$'s are the same, denoted by $n(z)$. Hence $z\in \overline{H^{\sigma_{n(z)}}_{{k_l}^{-1},S}.x},\forall l$. Since $\overline{H^{\sigma_{n(z)}}_{\epsilon,S}.x}$ is decreasing as $\epsilon$ decreases, $z$ is in the limit set $\bigcap_{\epsilon>0}\overline{H^{\sigma_{n(z)}}_{\epsilon,S}.x}$. Since $z\in M$ is chosen arbitrarily, this proves
$\bigcup_{n=1}^N\bigcap_{\epsilon>0}\overline{H^{\sigma_n}_{\epsilon,S}.x}\supset M$. So (\ref{minunion1}) holds.\\

{\noindent\bf Step 2}. Denote by $m\leq N$ the largest number such that there are $(H,S)$-invariant closed subsets $\Omega_1,\cdots,\Omega_N$ of $M$ that satisfy the following three conditions:
\begin{eqnarray}
\bullet&\text{At least }m\text { of the }\Omega_n\text{'s are minimal }(H,S)\text{-invariant sets};\\
\bullet&\label{minunion6}\bigcup_{n=1}^N\Omega_n=M;\\
\bullet&\label{minunion7}\exists l,\exists x\in\Omega_l\text{ s.t. }\forall n, \Omega_n=\bigcap_{\epsilon>0}\overline{H^{\sigma_n-\sigma_l}_{\epsilon,S}.x}.
\end{eqnarray}

To obtain the proposition, one needs to show $m=N$. We show first $m\geq 1$.

Since $M$ is $(\tilde H,S)$-invariant, hence $(H,S)$-invariant as well. By Lemma \ref{minZorn} there is a minimal $(H,S)$-invariant closed set $\Omega_1\subset M$. Take an arbitrary point $x\in\Omega_1$, then by Lemma \ref{mingene}, $\bigcap_{\epsilon>0}\overline{H_{\epsilon,S}.x}=\Omega_1$. Take $l=1$ and set $\Omega_n=\bigcap_{\epsilon>0}\overline{H^{\sigma_n}_{\epsilon,S}.x}$, which doesn't change the meaning of $\Omega_1$; this establishes (\ref{minunion7}). Then by Corollary \ref{cosetcomposlice}, $\Omega_1,\cdots,\Omega_N$ are all non-empty $(H,S)$-invariant closed sets. (\ref{minunion6}) follows from (\ref{minunion1}). As $\Omega_1$ is minimal, we see $m\geq 1$.\\

{\noindent\bf Step 3}. We now prove $m=N$.

Suppose $m< N$. Then among the corresponding $(H,S)$-invariant sets $\Omega_1,\cdots,\Omega_N$, there is at least one $\Omega_k$ which is not minimal. By Lemma \ref{minZorn} we may take a minimal $(H,S)$-invariant closed subset $\Omega'_k\subsetneq\Omega_k$. Pick $x'\in \Omega'_k$ then by Lemma \ref{mingene}, $\bigcap_{\epsilon>0}\overline{H_{\epsilon,S}.x'}=\Omega'_k$. In accordance with (\ref{minunion7}), define
\begin{equation}\Omega'_n=\bigcap_{\epsilon>0}\overline{H^{\sigma_n-\sigma_k}_{\epsilon,S}.x'}, \forall n\in\{1,\cdots,n\}.\end{equation} Remark this doesn't change the definition of $\Omega'_k$.

Since $x'\in\Omega_k=\bigcap_{\epsilon>0}\overline{H^{\sigma_k-\sigma_l}_{\epsilon,S}.x}$, it follows from Corollary \ref{cosetcomposlice}.(i) that \begin{equation}\Omega'_n\subset\bigcap_{\epsilon>0}\overline{H^{(\sigma_n-\sigma_k)+(\sigma_k-\sigma_l)}_{\epsilon,S}.x}=\bigcap_{\epsilon>0}\overline{H^{\sigma_n-\sigma_l}_{\epsilon,S}.x}=\Omega_n.\end{equation} On the other hand, by Corollary \ref{cosetcomposlice}.(ii), all the $\Omega'_n$'s are non-empty and $(H,S)$-invariant. By Corollary \ref{cosetsubgp}, $\sigma_n-\sigma_k+H$ runs through $\sigma_1+H,\cdots,\sigma_N+H$ as $n$ varies, so
\begin{equation}\bigcup_{n=1}^N\Omega'_n=\bigcup_{n=1}^N\bigcap_{\epsilon>0}\overline{H^{\sigma_n}_{\epsilon,S}.x'}=M,\end{equation}
where the second equality follows from (\ref{minunion1}). Thus the $\Omega'_n$'s verify (\ref{minunion6}).

For those $n$ such that $\Omega_n$ is a minimal $(H,S)$-invariant closed set, by minimality $\Omega'_n$ is equal to $\Omega_n$ hence is still minimal. However the minimal set  $\Omega'_k$ is not one of these. Therefore there are at least $m+1$ minimal sets among the $\Omega'_n$'s, which contradicts the maximality of $m$. Hence $m=N$, this completes the proof of proposition.\end{proof}

\section{Sets that accumulate at a $V_{\langle S\rangle}$-translated torsion point}\label{patterncase}

From now on let $\alpha$, $\zeta$, $X$, $S$, $\langle S\rangle$ and $V_{\langle S\rangle}$ be as in Theorem \ref{denseorcentral} and $H$ be a subgroup of finite index in $G=\bZ^r$.

\begin{definition}\label{pattern}For a real subspace $V$ of $\bR^{r_1}\oplus\bC^{r_2}$, we say a closed subset $\Omega\subset X$ contains a {\bf $V$-pattern} if there is a sequence of points $\{z_k\}_{k=1}^\infty\subset\Omega$ converging to a $V$-translated torsion point $z\in \Omega$, such that $z_k-z\notin V$ for all $k$.\end{definition}

Here and from now on in similar situations, the difference $z_k-z$ is viewed as a vector in $\bR^{r_1}\oplus\bC^{r_2}$ whose length tends to $0$ as $k\rightarrow\infty$, as $X$ is locally isomorphic to $\bR^{r_1}\oplus\bC^{r_2}$.

Throughout this section we consider a ``finitely generated'' $(H,S)$-invariant set \begin{equation}\label{fgnonisolated}A=\bigcap_{\epsilon>0}\overline{H_{\epsilon,S}.x_1}+\cdots+\bigcap_{\epsilon>0}\overline{H_{\epsilon,S}.x_l}\subset X,\end{equation} whose invariance follows from Corollary \ref{cosetcomposlice} and Lemma \ref{invpm}. By Lemma \ref{fginvariant}, $A=\bigcap_{\epsilon>0}A_\epsilon$ where  \begin{equation}\label{invnbhd}A_\epsilon=\overline{H_{\epsilon,S}.x_1}+\cdots+\overline{H_{\epsilon,S}.x_l}.\end{equation}
Furthermore, $H_{\epsilon-\epsilon_0\,S}.A_{\epsilon_0}\subset A_\epsilon$ for all $\epsilon>\epsilon_0>0$.

The rest of Section \ref{patterncase} will be devoted to the proof of the following proposition.

\begin{proposition}\label{nonisolated} Let $A$ and $A_\epsilon$ be as above. If $A_{\epsilon_0}$ contains a $V_{\langle S\rangle}$-pattern for some $\epsilon_0$, then $A_\epsilon=X$ for all $\epsilon>\epsilon_0$. In particular, if $A_\epsilon$ contains a $V_{\langle S\rangle}$-pattern for all positive $\epsilon$, then $A_\epsilon=X$, for all $\epsilon>0$ and thus $A=X$ as well.\end{proposition}

\subsection{Orbits of $V_{\langle S\rangle}$-translated torsion points} To begin with, we remark the ``only if'' part of  Theorem \ref{denseorcentral} is not difficult to see.

Actually,  suppose $x=x_*+v$ where $x_*\in X$ is a torsion point and $v\in V_{\langle S\rangle}$. Then the orbit $G.x_*$ is a finite set of torsion points (because $qx'=0$ for any point $x'\in G.x_*$ where $q$ denotes the order of $x_*$ and there are only finitely many torsion points of order $q$ in $X$).

By definition of $\langle S\rangle$, there are constants $c_{ij}\in\bR, \forall j\in\langle S\rangle, \forall i\in S$ such that $\lambda_j=\sum_{i\in S}c_{ij}\lambda_i$. Denote $c=\max_{j\in\langle S\rangle}\sum_{i\in S}|c_{ij}|$. Then for all $j\in\langle S\rangle$, \begin{equation}\label{lambdaspan}|\lambda_j(\bfn)|\leq\sum_{i\in S}|c_{ij}\lambda_i(\bfn)|\leq c\max_{i\in S}|\lambda_i(\bfn)|.\end{equation}

Assume  $|\lambda_i(\bfn)|<\epsilon, \forall i\in S$, then \begin{equation}\label{centraldeform}|\zeta_j^\bfn|\in(e^{-c\epsilon},e^{c\epsilon}),\forall j\in\langle S\rangle.\end{equation}

We can write $v=\sum_{j\in\langle S\rangle}v_j$ with $v_j\in V_j$. Then $\zeta^\bfn.v$ is in $V_{\langle S\rangle}$ and its $V_j$ coordinate is $\zeta_j^\bfn v_j$. Hence by (\ref{centraldeform}), $e^{-c\epsilon}|v|<|\zeta^\bfn.v|<e^{c\epsilon}|v|$.

So $\zeta^\bfn.x=\zeta^\bfn.x_*+\zeta^\bfn.v$ belongs to $D_\epsilon$ where \begin{equation}D_\epsilon=\{x'+v':x'\in G.x_*, v'\in V_{\langle S\rangle}, |v'|\leq e^{c\epsilon}|v|\}\end{equation} is a finite union of $V_{\langle S\rangle}$-discs centered at torsion points.

Furthermore, let $\epsilon<\frac1{c+1}$. If $\bfn\in G_{\epsilon,S}$, then in addition to (\ref{centraldeform}), $\|\beta_j(\bfn)\|<\epsilon,\forall j\in\langle S\rangle$. Note by definition $\|\beta_j(\bfn)\|=\Arg\zeta_j^\bfn$ where $\Arg$ denotes the principal value of complex argument, so \begin{equation}\begin{split}|\zeta_j^\bfn-1|=&|e^{\lambda_j(\bfn)+\imag\Arg\zeta_j^\bfn}-1|\leq 2|\lambda_j(\bfn)+\imag\Arg\zeta_j^\bfn|\\
\leq &2(c\epsilon+\epsilon)=2(c+1)\epsilon,\end{split}\end{equation} where we used the facts that $|\lambda_j(\bfn)+\imag\Arg\zeta_j^\bfn|\leq (c+1)\epsilon<1$ and $|e^z-1|\leq 2|z|$ as long as $|z|< 1$. Therefore $|\zeta_j^\bfn.v_j-v_j|\leq 2(c+1)\epsilon|v_j|$, $\forall j\in\langle S\rangle$. In consequence $\forall\bfn\in G_{\epsilon,S}$, \begin{equation}|\zeta^\bfn.v-v|\leq 2(c+1)\epsilon|v|,\end{equation} and thus $\zeta^\bfn.x\in N_\epsilon$ where \begin{equation}N_\epsilon=\{x'+v':x'\in G.\epsilon_*,v'\in V_{\langle S\rangle},|v'-v|<2(c+1)\epsilon|v|\}.\end{equation}

Because of finiteness of $G.x_*$, both $D_\epsilon$ and $N_\epsilon$ are closed in $X$. Moreover, notice $\bigcap_{\epsilon>0}N_\epsilon=\{x'+v:x'\in G.x_*\}$, which is finite.

Therefore, we have actually shown the following lemma:

\begin{lemma}\label{torsionorbit}Let $\zeta$, $X$, $I$, $S$, $\langle S\rangle$ and $V_{\langle S\rangle}$ be as in Theorem \ref{denseorcentral} and $x$ be a $V_{\langle S\rangle}$-translated torsion point in $X$. then for all $\epsilon>0$:
\begin{enumerate}
\item[(i)] The set (\ref{Xpartialorb}) is contained in a finite union of $V_{\langle S\rangle}$-discs centered at torsion points;
\item[(ii)] In addition, $\bigcap_{\epsilon>0}\overline{G_{\epsilon,S}.x}$ is a finite set of $V_{\langle S\rangle}$-translated torsion points.\end{enumerate}\end{lemma}

It can be shown that if the conditions in Theorem \ref{denseorcentral} are satisfied, then $\dim V_{\langle S\rangle}<d$.

\begin{lemma}\label{centralproper}In the setting of Theorem \ref{denseorcentral}, $\dim V_{\langle S\rangle}$ must be a proper subspace of $\bR^{r_1}\oplus\bC^{r_2}$.\end{lemma}

\begin{proof}Suppose the lemma fails, then $V_j\subset V_{\langle S\rangle}, \forall j\in I$, or equivalently, $\langle S\rangle=I$. So the functionals $\{\lambda_i:i\in I\}$ span $L_S$, which, by assumption in Theorem \ref{denseorcentral}, is a subspace of dimension at most $r-2$ in $(\bR^r)^*$. If $\bfe_1,\cdots,\bfe_r$ are a basis of $\bZ^r$, then it is equivalent to say that the rank of the matrix 
\begin{equation}M=\big(\lambda_i(\bfe_k)\big)_{i\in I;1\leq k\leq r}=\big(\log|\sigma_i(\zeta^{\bfe_k})|\big)_{i\in I;1\leq k\leq r}\end{equation} is at most $r-2$.

Recall $U_K$ is a finitely generated abelian group of rank $r_1+r_2-1$, so we can always extend $\zeta$ to a group embedding of $\bZ^{r_1+r_2-1}$ into $U_K$. In other words, if we regard $\bZ^r$ as a subgroup of $\bZ^{r_1+r_2-1}$ and supplement $\bfe_1,\cdots,\bfe_r$ by $\bfe_{r+1},\cdots,\bfe_{r_1+r_2-1}$ to form a basis of $\bZ^{r_1+r_2-1}$, then there are elements $\zeta^{\bfe_{r+1}},\cdots,\zeta^{\bfe_{r_1+r_2-1}}\in U_K$ such that $\zeta^{\bfe_1},\cdots,\zeta^{\bfe_{r_1+r_2-1}}$ generate a subgroup of rank $r_1+r_2-1$, i.e. of finite index, in $U_K$.

Consider the $(r_1+r_2)\times(r_1+r_2-1)$ matrix \begin{equation}\tilde M=\big(\log|\sigma_i(\zeta^{\bfe_k})|\big)_{i\in I;1\leq k\leq r_1+r_2-1}.\end{equation} Since $\tilde M$ has $M$ as a $(r_1+r_2)\times r$ submatrix, \begin{equation}\label{ranktoosmall}\begin{split}\rank\tilde M\leq &\rank M+(r_1+r_2-1)-r\\
\leq&(r-2)+(r_1+r_2-1)-r\\
=&r_1+r_2-3.
\end{split}\end{equation}

However, because $\eta^{\bfe_1},\cdots,\eta^{\bfe_{r_1+r_2-1}}$ generate a finite-index subgroup of $U_K$, the $\bZ$-span of the rows of $\tilde M$ has finite index in $\big\{\big(\log|\sigma_i(\theta)|\big)_{i\in I}:\theta\in U_K\}$, which is a discrete subgroup of rank $r_1+r_2-1$ in $\bR^I$ by Dirichlet's Unit Theorem. It follows that $\rank\tilde M=r_1+r_2-1$, which contradicts (\ref{ranktoosmall}). This completes the proof.\end{proof}

\begin{corollary}\label{central} In the setting of Theorem \ref{denseorcentral}, if $x\in X$ is a $V_{\langle S\rangle}$-translated torsion point, then the set (\ref{Xpartialorb}) is not dense in $X$.\end{corollary}
\begin{proof} By Lemma \ref{torsionorbit}, this set is in a finite union of compact $V_{\langle S\rangle}$-discs hence so is its closure. Thus the dimension of the closure is strictly less than $\dim X=d$ by Lemma (\ref{centralproper}), which implies the corollary.\end{proof}

Theorem \ref{denseorcentral'}.(ii) is covered by Lemma \ref{torsionorbit}. The rest of the paper will be focusing on the proof of Theorem \ref{denseorcentral'}.(i).

\subsection{Concentration to a coarse Lyapunov subspace} Before working on $X$, we study first how elements of $H_{\epsilon,S}$ act on the linear space $\bR^{r_1}\oplus\bC^{r_2}=\bigoplus_{i\in I}V_i$.

We show that $H_{\epsilon,S}$ is a good enough approximation to the hyperplane $P_S$ in (\ref{planenonhyp}), This uses the ideas already presented in \S\ref{suspension}.

\begin{lemma}\label{nongroupapprox}For all $\epsilon>0$, there is a constant $C=C(\epsilon,H)$ such that $\forall\eta\in P_S$, $\exists\bfn\in H_{\epsilon,S}$ such that $|\bfn-\eta|<C$.\end{lemma}
\begin{proof} Let $\Omega$ be as in (\ref{suslattice}), then  $\big(\bR^r\times(\bR/2\pi\bZ)^{r_1+r_2}\big)/\Omega\cong\bT^{r+r_1+r_2}$. Denote the natural projection by
\begin{equation}p:\bR^r\times(\bR/2\pi\bZ)^{r_1+r_2}\mapsto \big(\bR^r\times(\bR/2\pi\bZ)^{r_1+r_2}\big)/\Omega.\end{equation}

Take the product $P_S\times\{0\}$ where $0$ denotes the trivial vector in $(\bR/2\pi\bZ)^{r_1+r_2}$. Then $p(P_S\times\{0\})$ is a connected subgroup of $\big(\bR^r\times(\bR/2\pi\bZ)^{r_1+r_2}\big)/\Omega$ and therefore $\overline{p(P_S\times\{0\})}$ is a connected closed subgroup, which has to be a subtorus containing $0$. For any $\epsilon>0$, take an $\frac\epsilon2$-dense subset $E$ of $\overline{p(P_S\times\{0\})}$ (i.e. $\forall z\in\overline{p(P_S\times\{0\})}$, $\exists z'\in p(P_S\times\{0\})$, $\|z-z'\|<\frac\epsilon2$). By compactness of $\overline{p(P_S\times\{0\})}$, we can choose $E$ to be finite. By density of $p(P_S\times\{0\})$, we may slightly modify $E$ so that it is inside $p(P_S\times\{0\})$ and is $\epsilon$-dense in $\overline{p(P_S\times\{0\})}$. Suppose \begin{equation}E=\{p\big((\eta_k,0)\big): k=1,\cdots,m\}\end{equation} where $\eta_k\in P_S$. Let $C=\max_{l=1}^m|\eta_k|+\epsilon$. Then for all $\eta\in P_S$, there is an $\eta_k$ such that $\big\|p\big((\eta-\eta_k,0)\big)\big\|=\big\|p\big((\eta,0)\big)-p\big((\eta_k,0)\big)\big\|<\epsilon$. By construction of $\Omega$, $\exists\bfn\in H$ such that $|\eta-\eta_k-\bfn|<\epsilon$ and $\beta_j(\bfn)$ is within distance $\epsilon$ from zero modulo $2\pi$ for all $j\in I$.

It follows first that $|\eta-\bfn|<|\eta_k|+\epsilon\leq C$. Hence for all $i\in S$, by construction of $P_S$, $\lambda_i(\eta-\eta_k)=0$ and \begin{equation}|\lambda_i(\bfn)|\leq\|\lambda_i\|\cdot|\eta-\eta_k-\bfn|<\|\lambda_i\|\epsilon.\end{equation} Moreover, $\|\beta_j(\bfn)\|<\epsilon, \forall j\in I$. Therefore $\bfn\in H_{\max(\max_{i\in S}\|\lambda_i\|, 1)\epsilon,S}$. By replacing $\epsilon$ with $\frac\epsilon{\max(\max_{i\in S}\|\lambda_i\|, 1)}$, we obtain the lemma.\end{proof}

Now for every $i\in S$, denote by $\lambda_i^S\in (P_S)^*$ the restriction of the Lyapunov functional $\lambda_i$ to $P_S$. First, notice \begin{equation}\label{vanishSgene}\lambda_i^S\equiv 0\Leftrightarrow\lambda_i\in P_S^\bot=L_S\Leftrightarrow i\in\langle S\rangle,\end{equation} which was the reason to introduce $\langle S\rangle$ in Theorem \ref{denseorcentral}.

By grouping together $\lambda_i^S$'s that are positively proportional to each other, the set of indices $I$ decomposes into a disjoint union of subsets of the form
\begin{equation}I_{[\lambda]}=\{i\in I:\lambda_i^S\in\bR_+\lambda\},\end{equation} where $\lambda\in (P_S)^*$. There is a finite collection $\Lambda^S$ of equivalence classes of the form $[\lambda]=\bR_+\lambda$ such that for any $[\lambda]\in\Lambda^S$, $I_{[\lambda]}$ is not empty and \begin{equation}I=\bigsqcup_{[\lambda]\in\Lambda^S}I_{[\lambda]}.\end{equation}

\begin{definition}\label{cLs} The {\bf coarse Lyapunov subspace} associated to $[\lambda]\in\Lambda^S$ is \begin{equation}V_{[\lambda]}=\bigoplus_{i\in I_{[\lambda]}}V_i.\end{equation}

For a subset $\Lambda\subset\Lambda^S$, let \begin{equation}V_{\Lambda}=\bigoplus_{[\lambda]\in\Lambda}V_{[\lambda]}.\end{equation} \end{definition}

By (\ref{vanishSgene}), the Lyapunov subspace $V_{[0]}$ corresponding to the constant zero map in $(P_S)^*$ is exactly  the central foliation $V_{\langle S\rangle}$. By Lemma \ref{centralproper}, $\Lambda^S$ contains equivalence classes other than $[0]$.

\begin{proposition}\label{concentration}If $A_{\epsilon_0}$ contains a $V_{\langle S\rangle}$-pattern for some $\epsilon_0$, then $\exists[\lambda]\in\Lambda^S\backslash\{[0]\}$ such that $\forall \epsilon>\epsilon_0$, there is a sequence $\{y_k\}_{k=1}^\infty\subset A_\epsilon$ converging to a $V_{\langle S\rangle}$-translated torsion point $y$ such that $y_k-y\in (V_{\langle S\rangle}\oplus V_{[\lambda]})\backslash V_{\langle S\rangle}$ for all $\forall k$.\end{proposition}

Here again $y_k-y$ is viewed as a very short vector in $\bR^{r_1}\oplus\bC^{r_2}$.

\begin{proof}Let $\Lambda$ be a minimal non-empty subset of $\Lambda^S\backslash\{[0]\}$ verifying the following condition:\\

\begin{itemize}\item[($\star$)] {\it  $\forall\epsilon>\epsilon_0$, there is a sequence of points $\{y_k\}_{n=1}^\infty$ from $A_\epsilon$ that converges to a $V_{\langle S\rangle}$-translated torsion point $y$, which may depend on $\epsilon$, such that $y_k-y\in (V_{\langle S\rangle}\oplus V_{\Lambda})\backslash V_{\langle S\rangle}$.}\\\end{itemize}

The definition of $\Lambda$ makes sense since by definition of a $V_{\langle S\rangle}$-pattern, $\Lambda^S\backslash\{[0]\}$ satisfies condition ($\star$). To establish the proposition it suffices to prove any minimal set $\Lambda$ satisfying ($\star$) consists of a single element $[\lambda]\in\Lambda^S\backslash\{[0]\}$.

Assume for contradiction that $\Lambda$ consists of more than one elements.

Observe for any two non-zero linear maps $\lambda_1,\lambda_2\in (P_S)^*$, if $\lambda_2\notin\bR_+\lambda_1$ then there is $\eta\in P_S$ such that $\lambda_1(\eta)>0$, $\lambda_2(\eta)<0$. Thus $\Lambda$ decomposes as a disjoint union of two non-empty parts $\Lambda_+$ and $\Lambda_-$, and there exists $\eta\in P_S$ such that \begin{equation}\label{Lyacut}\lambda(\eta)>0\text{(resp. }<0\text{)},\ \forall[\lambda]\in \Lambda_+\text{(resp. }\Lambda_-\text{)}.\end{equation}

Let $\epsilon'>\epsilon>\epsilon_0$, take the sequence $\{y_k\}_{k=1}^\infty\in A_\epsilon$ and its limit $y$ in assumption ($\star$) with respect to parameter $\epsilon$, such that $y_k-y\in (V_{\langle S\rangle}\oplus V_{\Lambda})\backslash V_{\langle S\rangle}$ and $\theta_k=\|y_k-y\|\rightarrow 0$. Write  $y_k$ as $y+v_k^0+v_k^++v_k^-$ where:
\begin{itemize}
\item $v_k^0$, $v_k^+$ and $v_k^-$ are respectively vectors from  $V_{\langle S\rangle}$, $V_{\Lambda_+}$ and $V_{\Lambda_-}$;
\item $|v_k^0|$, $|v_k^+|$ and $|v_k^-|$ are all bounded by $\theta_k$;
\item at least one of $v_k^+$ and $v_k^-$ is not equal to zero.
\end{itemize}

By passing to a sequence of $\epsilon\rightarrow 0$ and a subsequence of $\{y_k\}$ for each of these $\epsilon$'s, we may assume without loss of generality that for all $\epsilon>\epsilon_0$, the sequence $\{y_k\}_{k=1}^\infty$ is chosen in such a way that the corresponding $v_k^+$ does not vanish for any $k$.

For each $k$, consider the map \begin{equation}\label{lengthgrowth}b_k^+(t)=\Big(\sum_{i\in\bigcup_{[\lambda]\in\Lambda_+}I_{[\lambda]}}(e^{\lambda_i(t\eta)}|(v_k^+)_i|)^2\Big)^\frac12\end{equation} from $\bR$ to $\bR_k^+$, where $(v_k^+)_i$ is the projection of $v_k^+$ in $V_i$. If $i\in\bigcup_{[\lambda]\in\Lambda_+}I_{[\lambda]}$ then by construction of coarse Lyapunov subspaces and (\ref{Lyacut}) $\lambda_i(\eta)=\lambda_i^S(\eta)>0$, hence $t$ is strictly increasing for $v_k^+\neq 0$. Similarly the map $b_k^-(t)=\Big(\sum_{i\in\bigcup_{[\lambda]\in\Lambda_-}I_{[\lambda]}}(e^{\lambda_i(t\eta)}|(v_k^-)_i|)^2\Big)^\frac12$ is decreasing (strictly decreasing if $v_k^-\neq 0$) and $b_k^0(t)=\Big(\sum_{i\in\langle S\rangle}(e^{\lambda_i(t\eta)}|(v_k^0)_i|)^2\Big)^\frac12\equiv |v_k^0|$.

Fix now a positive number $\delta$. For sufficiently large $k$, $\theta_k<\delta$ and $b_k^+(0)=|v_k^+|\leq\theta_k<\delta$. Hence $\exists t_k>0$ such that $b_k^+(t_k)=\delta$. Take $\bfn_k\in H_{\epsilon'-\epsilon,S}$ within distance $C$ from $t_k\eta$ where $
C=C(\epsilon'-\epsilon,H)$ is the constant given by Lemma \ref{nongroupapprox}.

Consider \begin{equation}\label{localdecomp}\zeta^{\bfn_k}.y_k=\zeta^{\bfn_k}.y+\zeta^{\bfn_k}.v_k^0+\zeta^{\bfn_k}.v_k^++\zeta^{\bfn_k}.v_k^-.\end{equation} As $X$ is compact, by passing to a subsequence one may assume $\zeta^{\bfn_k}.y$ converges as $k\rightarrow\infty$.

Note the length of $\zeta^{\bfn_k}v_k^+$ is $\Big(\sum_{i\in\bigcup_{[\lambda]\in\Lambda_+}I_{[\lambda]}}(e^{\lambda_i(\bfn_k)}|(v_k^+)_i|)^2\Big)^\frac12$. Since $|\lambda_i(\bfn_k)-\lambda_i(t_k\eta)|\leq\|\lambda_i\||\bfn_k-t_k\eta|<aC$ where $a=\max_{i\in I}{\|\lambda_i\|}$, we have the relation $\frac{|\zeta^{\bfn_k}.v_k^+|}{b_k^+(t_k)}\in(e^{-aC},e^{aC})$. Similarly $\frac{|\zeta^{\bfn_k}.v_k^-|}{b_k^-(t_k)}$, $\frac{|\zeta^{\bfn_k}.v_k^0|}{b_k^0(t_k)}$ fall into the same range. Hence \begin{equation}\label{fixedsmallexpansion}|\zeta^{\bfn_k}.v_k^+|\in (e^{-aC}b_k^+(t_k),e^{aC}b_k^+(t_k))=(e^{-aC}\delta,e^{aC}\delta).\end{equation} Moreover \begin{equation}|\zeta^{\bfn_k}.v_k^-|\leq e^{aC}b_k^-(t_k)\leq e^{aC}b_k^-(0)=e^{aC}|v_k^-|\leq e^{aC}\theta_k,\end{equation} and similarly \begin{equation}|\zeta^{\bfn_k}.v_k^0|\leq e^{aC}b_k^0(t_k)=e^{aC}|v_k^0|\leq e^{aC}\theta_k.\end{equation}

Note $a,C$ are independent of $\delta$ and $\theta_k$. So since $\theta_k\rightarrow 0$, $\zeta^{\bfn_k}.v_k^-$ and $\zeta^{\bfn_k}.v_k^0$ converge to $0$ as $k\rightarrow\infty$. By Lemma \ref{torsionorbit}, the limit of $\zeta^{\bfn_k}.y$ is a $V_{\langle S\rangle}$-translated torsion point $z_\delta\in\overline{G_{\epsilon'-\epsilon,S}.y}$ which depends on $\delta$. Again as $\zeta^{\bfn_k}.v_k^+$ is in the ball of fixed radius $e^{aC}\delta$ inside $V_{\Lambda_+}$, it is all right to suppose $\lim_{k\rightarrow\infty}\zeta^{\bfn_k}.v_k^+$ exists by passing to a subsequence if necessary. Denote this limit by $w_\delta$, then by (\ref{fixedsmallexpansion}), \begin{equation}\label{reducedcLS}w_\delta\in V_{\Lambda_+},\ |w_\delta|\in [e^{-aC}\delta,e^{aC}\delta].\end{equation}
Take limit of (\ref{localdecomp}) by summing up the terms, we see that
\begin{equation}\label{convergeoutside}z_\delta+w_\delta=\lim_{k\rightarrow\infty}\zeta^{\bfn_k}.y_k\in\overline{H_{\epsilon'-\epsilon,S}.A_\epsilon}\subset A_{\epsilon'},\end{equation} where Lemma \ref{fginvariant} is used in the last step.

It was shown in Lemma \ref{torsionorbit} that $\overline{G_{\epsilon'-\epsilon,S}.y}$ is inside a finite union of compact $V_{\langle S\rangle}$-discs centered at torsion points. Take a sequence of positive numbers $\{\delta_h\}_{h=1}^\infty$ decaying to $0$ such that the $z_{\delta_h}$'s are in the same $V_{\langle S\rangle}$-disc and converge to some point $y'$. Then $y'$ belongs to the same disc and is thus a $V_{\langle S\rangle}$-translated torsion point. As the constant $aC$ is independent of $\delta$,  $w_{\delta_h}\rightarrow 0$ by (\ref{reducedcLS}). In consequence, the points $y'_h=z_{\delta_h}+w_{\delta_h}$ converge to $y$ as well. Moreover $z_{\delta_h}-y'\in V_{\langle S\rangle}$ and $w_{\delta_h}\in V_{\Lambda_+}\backslash\{0\}$, so the difference $y'_h-y$ is in $(V_{\langle S\rangle}\oplus V_{\Lambda_+})\backslash V_{\langle S\rangle}$.

By (\ref{convergeoutside}), $y'_h\in A_{\epsilon'}$. So we have proved for any $\epsilon'>\epsilon_0$, $A_{\epsilon'}$ contains a sequence $\{y'_h\}_{h=1}^\infty$ and a $V_{\langle S\rangle}$-translated point $y'$ satisfying condition ($\star$) with respect to the non-empty subset $\Lambda_+\subsetneq\Lambda$; a contradiction to the minimality of $\Lambda$. This completes the proof.
\end{proof}

The proposition will be used later in the form of the next corollary.

\begin{corollary}\label{longcLsvector}If $A_{\epsilon_0}$ contains a $V_{\langle S\rangle}$-pattern for some $\epsilon_0$, then there exists $[\lambda]\in\Lambda^S\backslash\{[0]\}$ such that $\forall \epsilon>\epsilon_0$, there is a constant $C=C(\epsilon_0,\epsilon,H,A)$ such that for all $R>0$, $A_\epsilon$ contains a point of the form $y_*+v_{\langle S\rangle}+v$, where $y_*$ is a torsion point, $v_{\langle S\rangle}\in V_{\langle S\rangle}$, $v\in V_{[\lambda]}$ and $|v_{\langle S\rangle}|\leq C$, $v\in[R,CR]$.\end{corollary}

\begin{proof}By proposition, there is a class $[\lambda]\in\Lambda^S\backslash\{[0]\}$ such that we may pick from $A_{\frac{\epsilon_0+\epsilon}2}$ a point of the form $y_*'+v'_{\langle S\rangle}+v'$ where $y_*'$ is a torsion point, $v'_{\langle S\rangle}\in V_{\langle S\rangle}$, $v'\in V_{[\lambda]}$ with $|v'_{\langle S\rangle}|\leq C_1(\epsilon_0,\epsilon,H,A)$ and $0<|v'|\leq R$ where $C_1$ is a constant independent of $R$.

Since $\lambda\neq 0$, $\exists\eta\in P_S$ such that $\lambda(\eta)>0$. Similar to (\ref{lengthgrowth}), define \begin{equation}b(t)=\Big(\sum_{i\in\bigcup_{[\lambda]\in\Lambda'_+}I_{[\lambda]}}(e^{\lambda_i(t\eta)}|(v')_i|)^2\Big)^\frac12\end{equation} Then $b(0)=|v|$. And since each $\lambda_i$ involved is positively proportional to $\lambda$ on $P_S$, $\lambda_i(t\eta)>0$; so $b(t)$ is strictly increasing. Fix $t$ such that $b(t)=e^{aC_2}R$, where $a=\max_{i\in I}\|\lambda_i\|$ and $C_2=C_2(\frac{\epsilon-\epsilon_0}2,H)$ is the constant defined in Lemma \ref{nongroupapprox}, according to which  $\exists\bfn\in H_{\frac{\epsilon-\epsilon_0}2,S}\subset\bZ^r$ within distance $C_2$ from $\eta$.

Consider the point $\zeta^\bfn.y_*'+\zeta^\bfn.v'_{\langle S\rangle}+\zeta^\bfn.v'=\zeta^\bfn.(y_*'+v'_{\langle S\rangle}+v')\in H_{\frac{\epsilon-\epsilon_0}2,S}.A_{\frac{\epsilon_0+\epsilon}2}\subset A_\epsilon$. Obviously $\zeta^\bfn.y_*'$ is of torsion, $\zeta^\bfn.v'_{\langle S\rangle}\in V_{\langle S\rangle}$ and $\zeta^\bfn.v'\in V_{[\lambda]}$.

For any $j\in \langle S\rangle$, by inequality (\ref{centraldeform}) in the proof of Lemma \ref{torsionorbit}, $|\zeta_j^\bfn|\leq e^{\frac{c(\epsilon-\epsilon_0)}2}$ where $c$ is a constant depending only on the $\lambda_i$'s. Hence $|\zeta^\bfn.v'_{\langle S\rangle}|\leq e^{\frac{c(\epsilon-\epsilon_0)}2}C_1$.

Moreover $|\zeta^\bfn.v'|=\big(\sum_{i\in\bigcup_{[\lambda]\in\Lambda'_+}I_{[\lambda]}}(e^{\lambda_i(\bfn)}|(v')_i|)^2\big)^\frac12$, so similar to (\ref{fixedsmallexpansion}) we have $\frac{|\zeta^\bfn.v'|}{b(t)}\in (e^{-\max_{i\in I}|\lambda_i(\bfn-t\eta)|},e^{\max_{i\in I}|\lambda_i(\bfn-t\eta)|})\subset[e^{-aC_2},e^{aC_2}]$. Thus since $b(t)=e^{aC_2}R$, $|\zeta^\bfn.v'|\in [R,e^{2aC_2}R]$.

The lemma follows by setting $C=\max(e^{\frac{c(\epsilon-\epsilon_0)}2}C_1,e^{2aC_2})$.
\end{proof}

\subsection{Existence of arbitrarily long line segments}

We aim to prove the following:

\begin{proposition}\label{line}If $A_{\epsilon_0}$ contains a $V_{\langle S\rangle}$-pattern then there is a coarse Lyapunov subspace $V_{[\lambda]}$, where $[\lambda]\in\Lambda^S\backslash\{[0]\}$, satisfying that for all $\epsilon>\epsilon_0$, there exist a real line $L\subset V_{[\lambda]}$ that passes through the origin and a point $y\in X$, such that $A_\epsilon$ contains $y+L=\{y+v:v\in L\}\subset X$.\end{proposition}

To begin with, we remark that to establish the proposition, it suffices to show the next lemma.

\begin{definition}In a metric space, a subset $\Omega'$ is a {\bf $\delta$-net} of another subset $\Omega$ if for all $x\in \Omega$, there is $x'\in\Omega'$ within distance $\delta$ from $x$.\end{definition}

\begin{lemma}\label{netlongline}If $A_{\epsilon_0}$ contains a $V_{\langle S\rangle}$-pattern then there is a coarse Lyapunov subspace $V_{[\lambda]}$, where $[\lambda]\in\Lambda^S\backslash\{[0]\}$, satisfying: $\forall\epsilon>\epsilon_0$ and $\forall\delta,R>0$, there is a line segment $E\subset V_{[\lambda]}$ of length $R$ through the origin and a point $y\in X$ such that $A_\epsilon$ contains a $\delta$-net of $y+E=\{y+v:v\in E\}$.\end{lemma}

\begin{proof}[Proof of the implication Lemma \ref{netlongline} $\Rightarrow$ Proposition \ref{line}] Fix $\epsilon>\epsilon_0$ and any increasing sequence of positive numbers $R_k\rightarrow\infty$, by Lemma \ref{netlongline} there is a sequence of pairs $\{(y_k, w_k)\}_{k=1}^\infty$ where $y_k\in X$ and $w_k\in\bS V_{[\lambda]}:=\{w\in V_{[\lambda]}: |w|=1\}$ such that $\forall\rho\in[-R_k,R_k]$, there is a point $x\in A_\epsilon$ within distance $\frac1{R_k}$ from $y_k+\rho w_k$.

Passing to a subsequence if necessary, one may assume $y_k\rightarrow y\in X$ and $w_k\rightarrow w\in \bS V_{[\lambda]}$ since $X$ and $\bS V_{[\lambda]}$ are both compact. Let $L=\{\rho w:\rho\in\bR\}$.

For all $\rho>0$ and $\delta>0$, pick a sufficiently large $k$ so that $R_k>\max(|\rho|,\frac 3\delta)$, $\|y_k-y\|<\frac\delta3$ and $|w_k-w|<\frac\delta{3|\rho|}$. There is a point $x\in A_\epsilon$ within distance $\frac1{R_k}$ from $y_k+\rho w_k$, then
\begin{equation}\begin{split}&\|x-(y+\rho w)\|\\
\leq&\|x-(y_k+\rho w_k)\|+\|(y_k+\rho w_k)-(y+\rho w)\|\\
\leq &\frac1{R_k}+\|y-y_k\|+\rho|w_k-w|\leq\frac\delta3+\frac\delta3+|\rho|\cdot\frac\delta{3|\rho|}\\
=&\delta.\end{split}\end{equation}

Because $A_\epsilon$ is closed, by letting $\delta$ approach $0$ we see $y+\rho w\in A_\epsilon$ for all $\rho\in\bR$. This proves Proposition \ref{line}.\end{proof}

Before proving Lemma \ref{netlongline}, we explore some consequences to the rank assumption in Theorem \ref{denseorcentral}.

Take the subspace $L_S\subset(\bR^r)^*$ spanned by $\{\lambda_i: i\in S\}$, whose dimension we denote by $r_S$. Then $r_S\leq r-2$ by assumption (1) in Theorem \ref{denseorcentral}. Choose $i_1,\cdots,i_{r_S}\in S$ such that $\lambda_{i_1},\cdots,\lambda_{i_{r_S}}$ span $L_S$.

Let $V_{[\lambda]}$ be the coarse Lyapunov subspace in Proposition \ref{concentration} and Corollary \ref{longcLsvector}. Notice though $\langle S\rangle$ may be empty, $I_{[\lambda]}$ is not. Fix an arbitrary index $i_0\in I_{[\lambda]}$.

\begin{lemma}\label{jointspan}For an arbitrarily fixed $i_0\in I_{[\lambda]}$, $\lambda_{i_1},\cdots,\lambda_{i_{r_S}}$ and $\lambda_{i_0}$ are linearly independent. Moreover, for any $i\in \langle S\rangle\cup I_{[\lambda]}$, the Lyapunov functional $\lambda_i$ is a linear combination of $\lambda_{i_1},\cdots,\lambda_{i_{r_S}}$ and $\lambda_{i_0}$.\end{lemma}
\begin{proof}The linear independence is clear by the choice of $\lambda_{i_1},\cdots,\lambda_{i_{r_S}}$ and the fact that $i_0\notin\langle S\rangle$. For $i\in\langle S\rangle$, the lemma follows from the construction of $\langle S\rangle$ in Theorem \ref{denseorcentral}. Suppose $i\in V_{[\lambda]}$, by definition of coarse Lyapunov subspaces, $\exists c>0$ such that $\lambda_i|_{P_S}=\lambda_i^S=c\lambda_{i_0}^S=c\lambda_{i_0}|_{P_S}$. In other words $\lambda_i-c\lambda_{i_0}$ vanishes when restricted to $P_S$, or $\lambda_i-c\lambda_{i_0}\in P_S^\bot=L_S$. So $\lambda_i-c\lambda_{i_0}$ is in the linear span of $\lambda_{i_1},\cdots,\lambda_{i_{r_S}}$, which completes the proof.\end{proof}

Consider the following group morphism from $G=\bZ^r$ to the additive group $\bR^{r_S+1}\oplus(\bR/2\pi\bZ)^{r_1+r_2}$:
\begin{equation}\label{Logmap}\cL(\bfn)=\big(\lambda_{i_1}(\bfn),\cdots,\lambda_{i_{r_S}}(\bfn),\lambda_{i_0}(\bfn)\big)\oplus\big(\beta_j(\bfn)\big)_{j\in I}\end{equation}

The first observation is $\cL$ is injective. Actually suppose $\cL(\bfn)=0$ for some $\bfn$. Then both $\lambda_{i_0}(\bfn)=\log|\zeta_{i_0}^\bfn|\in\bR$ and $\beta_{i_0}(\bfn)\in\bR/2\pi\bZ$ vanish, so $\zeta_{i_0}^\bfn=1$. However $\zeta_{i_0}^\bfn$ is an algebraic conjugate to $\zeta^\bfn$. So $\zeta^\bfn=1$. But because $\zeta$ is a group embedding from $\bZ^r$ into $U_K$ this implies $\bfn=0$, which shows the injectivity of $\cL$.

In addition, we claim $0$ is a non-isolated point in $\cL(H')\subset\bR^{r_S+1}\oplus(\bR/2\pi\bZ)^{r_1+r_2}$ for any subgroup $H'<G$ of finite index. Suppose for contradiction that there is a ball $B_{2\theta}(0)$ of small radius $2\theta$ centered at $0$ such that $B_{2\theta}(0)\cap\cL(H')=\{0\}$. Then since $\cL(H')$ is an additive group, for all $\omega\in \bR^{r_S+1}\oplus(\bR/2\pi\bZ)^{r_1+r_2}$ the ball $B_\theta(\omega)$ centered at $\omega$ contains at most one element of $\cL(H')$. Observe the compact region $B_T=[-T,T]^{r_S+1}\oplus(\bR/2\pi\bZ)^{r_1+r_2}\subset \bR^{r_S+1}\oplus(\bR/2\pi\bZ)^{r_1+r_2}$ can be covered by $O_\theta(T^{r_S+1})$ balls of radius $\theta$. However on the other hand, since all the $\lambda_i$'s are linear, the image $\cL(\bfn)$ is inside $B_T(0)$ if $T\geq a|\bfn|$ where $a=\max_{i\in I}\|\lambda_i\|$. By pigeonhole principle, there are at most $O_\theta(T^{r_S+1})$ vectors $\bfn\in H$ of length $|\bfn|\leq\frac Ta$. But this becomes false for sufficiently large $T$ since by assumption (1) in Theorem \ref{denseorcentral}, $H$ is of rank $r\geq r_S+2$. Hence we obtain a contradiction and proved that $0$ must be non-isolated.

Because $\cL(H')$ is a subgroup, so is its closure $\overline{\cL(H')}$. The identity component $\cL^0$ of the closure is a connected closed subgroup of the abelian Lie group $\bR^{r_S+1}\oplus(\bR/2\pi\bZ)^{r_1+r_2}$. It is known that such a subgroup must be isomorphic to some $\bR^{d_1}\oplus\bT^{d_2}$. By non-isolatedness, $\cL^0$ is locally isomorphic to $\bR^q$ for some $q\geq 1$.

Moreover, any neighborhood of identity in $\cL^0$ is not contained in the closed subgroup $\{\lambda_{i_0}=0,\beta_{i_0}=0\}\subset\bR^{r_S+1}\oplus(\bR/2\pi\bZ)^{r_1+r_2}$, where $\lambda_{i_0}$ and $\beta_{i_0}$ refer respectively to the two coordinates corresponding to $\lambda_{i_0}(\bfn)$ and $\beta_{i_0}(\bfn)$ under the map $\cL$. This is because otherwise by the non-isolatedness of $0$ in $\cL(H')$, there is some $\bfn\neq\bf 0$ with $\lambda_{i_0}(\bfn)=0$, $\beta_{i_0}(\bfn)=0$, which would again imply $\zeta_{i_0}^\bfn=1$. In consequence $\bfn=\mathbf 0$, contradiction.

Combining these, there must be a real vector $\dot\lambda=(\dot\lambda_{i_1},\cdots,\dot\lambda_{i_{r_S}},\dot\lambda_{i_0})\oplus(\dot\beta_j)_{j\in I}$ such that at least one of $\dot\lambda_{i_0}$ and $\dot\beta_{i_0}$ does not vanish, and $\cL^0$ contains the projection of the line $\bR\dot\lambda\subset \bR^{r_S+1}\oplus\bR^{r_1+r_2}$ into $\bR^{r_S+1}\oplus(\bR/2\pi\bZ)^{r_1+r_2}$.

Recall if $V_i\cong\bR$ then $\beta_j(\bfn)$ is either $0$ or $\pi$ from $\bR/2\pi\bZ$ and thus $\cL(\bZ^r)$ is contained in two disjoint sets $\{\beta_j=\pi\}$ and $\{\beta_j=0\}$. As the identity component of $\overline{\cL(H')}$, $\cL^0$ must be contained in $\{\beta_j=0\}$. Thus \begin{equation}\label{realarg}\dot\beta_j=0\text{ if }V_i\cong\bR.\end{equation}

It is all right to suppose \begin{equation}\label{Lognormalize}\dot\lambda_{i_0}^2+\dot\beta_{i_0}^2=1\text{ and }|\dot \lambda|\leq C,\end{equation} where $C=C(H')$ can be made to be dependent only of the group action and the subgroup $H'$ because there are only finitely many possible choices of $[\lambda]$ and $i_1,\cdots,i_{r_S},i_0$.

Now we give a proof to Lemma \ref{netlongline}.

\begin{proof}[Proof of Lemma \ref{netlongline}] Suppose $A_{\epsilon_0}$ contains a $V_{\langle S\rangle}$-pattern. Fix $\epsilon>\epsilon_0$ and $R>0$, $\delta>0$.

Without loss of generality, we may assume \begin{equation}\label{netlonglineWLOG}\delta<\epsilon-\epsilon_0<1, R\geq 1.\end{equation}

Let $[\lambda]$ be given by Corollary \ref{longcLsvector}.

By Corollary \ref{longcLsvector}, inside $A_{\frac{\epsilon_0+\epsilon}2}$ there is a point of the form $y=y_*+v_{\langle S\rangle}+v$ where $y_*$ is of torsion, $v_{\langle S\rangle}\in V_{\langle S\rangle}$ is of length no more than $C_1$, and $v\in V_{[\lambda]}$ with
\begin{equation}\label{choicecLslength}|v|\in [\frac{3d^2C_1C_2^2R^2}\delta, \frac{3d^2C_1^2C_2^2R^2}\delta].\end{equation} Here $C_1=C_1(\epsilon_0,\epsilon,H,A)\geq1$ and $C_2$ is a constant to be specified later.

Since $\langle S\rangle\cap I_{[\lambda]}=\emptyset$, without causing ambiguity, decompose \begin{equation}\label{Slambdadecomp}v_{\langle S\rangle}=\sum_{i\in\langle S\rangle}v_i;\ v=\sum_{i\in I_{[\lambda]}}v_i\end{equation} where $v_i\in V_i$.

There is an index $i_0\in I_{[\lambda]}$ such that $v_{i_0}$ has greater length than any other $v_i$ with $i\in I_{[\lambda]}$. Then \begin{equation}\label{cLsdominant}|v_{i_0}|\in[\frac {3dC_1C_2^2R^2}\delta,\frac{3d^2C_1^2C_2^2R^2}\delta],\end{equation}
as $|I_{[\lambda]}|<d$ where $d$ is the dimension of the torus we work on.

Choose $i_1,\cdots,i_{r_S}\in S$ such that $\lambda_{i_1},\cdots,\lambda_{i_{r_S}}$ generate $L_S$. We may apply Lemma \ref{jointspan} and make the construction (\ref{Logmap}). Let \begin{equation}H'=\{\bfn\in H:\zeta^\bfn.y_*=y_*\},\end{equation} which has finite index in $H$, hence in $\bZ^r$ as well.

By the discussion preceding this proof, there exists \begin{equation}\dot\lambda=(\dot\lambda_{i_1},\cdots,\dot\lambda_{i_{r_S}},\dot\lambda_{i_0})\oplus(\dot\beta_j)_{j\in I}\in\bR^{r_S+1}\oplus\bR^{r_1+r_2}\end{equation} with $\dot\lambda_{i_0}^2+\dot\beta_{i_0}^2=1$ and $|\dot\lambda|\leq C_3$, such that $\overline{\cL(H')}$ contains the projection of the line $\bR\dot\lambda$ to $\bR^{r_S+1}\oplus(\bR/2\pi\bZ)^{r_1+r_2}$. Here $C_3\geq 1$ depends only on $H'$, which is the stabilizer of $y_*$ in $H$, hence eventually on $\epsilon_0$, $\epsilon$, $H$ and $A$.

By Lemma \ref{jointspan}, there is a unique decomposition \begin{equation}\label{jointspancoeff}\lambda_j=\sum_{i\in i_1,\cdots,i_{r_S},i_0}c_{ij}\lambda_i,\ \forall j\in\langle S\rangle\cup I_{[\lambda]}\end{equation} for some constants $c_{ij}\in\bR$. W

Corresponding to (\ref{jointspancoeff}), for each $j\in\langle S\rangle\cup I_{[\lambda]}$, set \begin{equation}\label{tildejointspancoeff}\dot\lambda_j=\sum_{i=i_1,\cdots,i_{r_S},i_0}c_{ij}\dot\lambda_i, \end{equation} which is compatible with the original value of $\dot\lambda_j$ if $j\in\{i_1,\cdots,i_{r_S},i_0\}$.

Notice if $j\in \langle S\rangle\cup I_{[\lambda]}$ then, \begin{equation}|\dot\lambda_j|\leq\left(\sum_{i\in\{i_1,\cdots,i_{r_S},i_0\}}|c_{ij}|\right)\cdot\max_{i\in\{i_1,\cdots,i_{r_S},i_0\}}|\dot\lambda_i|.\end{equation} Therefore, $\forall j\in \langle S\rangle\cup I_{[\lambda]}$,
\begin{equation}\begin{split}&|\dot\lambda_j+\imag\dot\beta_j|\\
\leq & \max\left(\sum_{i\in \{i_1,\cdots,i_{r_S},i_0\}}|c_{ij}|,1\right)\left(\max_{i\in\{i_1,\cdots,i_{r_S},i_0\}}|\dot\lambda_i|^2+|\dot\beta_j|^2\right)^\frac12\\
\leq&\left(\max_{j\in\langle S\rangle\cup I_{[\lambda]}}\sum_{i\in \{i_1,\cdots,i_{r_S},i_0\}}|c_{ij}|\right)\cdot|\dot\lambda|
\end{split}\end{equation}
Here in the second step we used the fact that \begin{equation}\max_{j\in\langle S\rangle\cup I_{[\lambda]}}\sum_{i\in \{i_1,\cdots,i_{r_S},i_0\}}|c_{ij}|\geq 1,\end{equation} which is true because when $j\in\{i_0,i_1,\cdots,i_{r_S}\}$ it is easy to see $c_{ij}$ equals $1$ if $i=j$ and vanishes otherwise.

Denote \begin{equation}C_2=C_3\cdot \max_{j\in\langle S\rangle\cup I_{[\lambda]}}\sum_{i\in \{i_1,\cdots,i_{r_S},i_0\}}|c_{ij}|,\end{equation} then $C_2\geq C_3\geq 1$ and it is determined by $\epsilon_0$, $\epsilon$, $H$ and $A$. Furthermore, by the argument above, \begin{equation}|\dot\lambda_j+\imag\dot\beta_j|\leq C_2, \forall j\in \langle S\rangle\cup I_{[\lambda]}.\end{equation}

Now set \begin{equation}\label{longlineslope}w=\sum_{j\in I_{[\lambda]}}(\dot\lambda_j+\imag\cdot\dot\beta_j)v_j\in V_{[\lambda]}.\end{equation}
Notice this definition makes sense because of (\ref{realarg}). Let $E$ be the line segment $\{y+\rho w:\rho\in[-\frac\delta {6dC_1C_2^2R},\frac\delta {6dC_1C_2^2R}]\}$.

We claim that:\\

\begin{itemize}
\item[(i)] $E$ has length greater than or equal to $R$;
\item[(ii)] $A_\epsilon$ contains a $\delta$-net of $E$.\\
\end{itemize}

\noindent{\bf Proof of (i)} The length of $E$ is $\frac\delta{3dC_1C_2^2R} |w|$. But \begin{equation}|w|\geq |(\dot\lambda_{i_0}+\imag\cdot\dot\beta_{i_0})v_{i_0}|=|\dot\lambda_{i_0}+\imag\cdot\dot\beta_{i_0}|\cdot|v_{i_0}|=|v_{i_0}|.\end{equation} By (\ref{cLsdominant}), length of $E\geq \frac\delta{3dC_1C_2^2R}\cdot\frac {3dC_1C_2^2R}\delta^2=R$.

\noindent{\bf Proof of (ii)} For all $\rho\in[-\frac\delta {6dC_1C_2^2R},\frac\delta {6dC_1C_2^2R}]$, we hope to find in the subset \begin{equation}H_{\frac{\epsilon-\epsilon_0}2,S}.(y_*+v_{\langle S\rangle}+v)\subset G_{\frac{\epsilon-\epsilon_0}2,S}.A_{\frac{\epsilon_2+\epsilon}2}\subset A_\epsilon\end{equation} a point arbitrarily close to $y+\rho w$. By the construction of $\dot\lambda$, $\forall\theta>0$, we may choose $\bfn\in H'$ such that $\cL(\bfn)$ is within distance $\theta$ from the projection of $\rho\dot\lambda\in\bR^{r_S+1}\oplus\bR^{r_1+r_2}$ to $\bR^{r_S+1}\oplus(\bR/2\pi\bZ)^{r_1+r_2}$. Therefore $\forall j\in\langle S\rangle\cup I_{[\lambda]}$, $\|\beta_j(\bfn)-\rho\dot\beta_j\|\leq \theta$; furthermore it follows from (\ref{jointspancoeff}) and (\ref{tildejointspancoeff}) that $\log|\zeta_j^\bfn|=\lambda_i(\bfn)=\sum_{i\in i_1,\cdots,i_{r_S},i_0}c_{ij}\lambda_i(\bfn)$ is within distance $(\max_{j\in\langle S\rangle\cup I_{[\lambda]}}\sum_{i\in \{i_1,\cdots,i_{r_S},i_0\}}|c_{ij}|)\cdot\theta$ from $\rho\dot\lambda_j$. Thus as $\theta$ can be arbitrarily small, we can choose $\bfn$ to make $\zeta_i^\bfn$ arbitrarily close to $e^{\rho(\dot\lambda_j+\imag\cdot\dot\beta_j)}$  simultaneously for all $j\in\langle S\rangle\cup I_{[\lambda]}$.

In particular, observe $\bfn\in (H')_{\frac{\epsilon-\epsilon_0}2,S}\subset H_{\frac{\epsilon-\epsilon_0}2,S}$. Actually, for all $i\in S$ and $j\in I$, because of the way we chose $\bfn\in H'$, the Lyapunov exponent $\lambda_i(\bfn)$ and the complex argument $\beta_j(\bfn)$ are respectively arbitrarily close to $\rho\dot\lambda_i$ and $\rho\dot\beta_i$. As both $|\rho\dot\lambda_i|$ and  $|\rho\dot\beta_j|$ are bounded by $\frac\delta {6dC_1C_2^2R}\cdot|\dot\lambda|\leq \frac\delta {6dC_1C_2^2R}\cdot C_2\leq \frac\delta {6dC_1C_2R}\leq\frac{\epsilon-\epsilon_0}6$ by (\ref{netlonglineWLOG}), \begin{equation}|\lambda_i(\bfn)|<\frac{\epsilon-\epsilon_0}2,\forall i\in S;\ |\beta_j(\bfn)|<\frac{\epsilon-\epsilon_0}2, ,\forall j\in I;\end{equation} that is, $\bfn\in (H')_{\frac{\epsilon-\epsilon_0}2,S}$.

Remark
\begin{equation}\begin{split}\zeta^\bfn.y=&\zeta^\bfn.y_*+\zeta^\bfn.v_{\langle S\rangle}+\zeta^\bfn.v\\
=&y_*+\sum_{j\in \langle S\rangle}\zeta_j^\bfn v_j+\sum_{j\in I_{[\lambda]}}\zeta_j^\bfn v_j,\end{split}\end{equation}
by the decomposition of $v_{\langle S\rangle}$ and $v$ and the fact that $\bfn$ is in the stablizer $H'$ of $y_*$.

In consequence since $\zeta_j^\bfn$ is arbitrarily close to $e^{\rho(\dot\lambda_j+\imag\cdot\dot\beta_j)}$, in order to prove (ii) it suffices to show both \begin{equation}\label{Sapprox}\left|\sum_{j\in \langle S\rangle}e^{\rho(\dot\lambda_j+\imag\cdot\dot\beta_j)} v_j-v_{\langle S\rangle}\right|\end{equation} and
\begin{equation}\label{lambdaapprox}\left|\sum_{j\in I_{[\lambda]}}e^{\rho(\dot\lambda_j+\imag\cdot\dot\beta_j)} v_j-v-\rho w\right|\end{equation} are bounded by $\frac\delta3$.

Remark while $|\alpha|\leq1$, $|e^\alpha-1|<2|\alpha|$ and $|e^\alpha-1-\alpha|<|\alpha|^2$.  By (\ref{netlonglineWLOG}), $|\rho(\dot\lambda_j+\imag\cdot\dot\beta_j)|\leq\frac\delta {6dC_1C_2^2R}\cdot C_3\leq\frac\delta {6dC_1C_2^2R}\cdot C_2=\frac\delta{6dC_1C_2R}\leq 1$. So we have \begin{equation}|e^{\rho(\dot\lambda_j+\imag\cdot\dot\beta_j)}-1|\leq 2|\rho(\dot\lambda_j+\imag\cdot\dot\beta_j)|\leq \frac\delta {3dC_1C_2R}\end{equation} and
\begin{equation}|e^{\rho(\dot\lambda_j+\imag\cdot\dot\beta_j)}-1-\rho(\dot\lambda_j+\imag\cdot\dot\beta_j)|\leq |\rho(\dot\lambda_j+\imag\cdot\dot\beta_j)|^2\leq(\frac\delta {6dC_1C_2R})^2.\end{equation}

Thus by (\ref{choicecLslength}), (\ref{Slambdadecomp}) and (\ref{longlineslope}), \begin{equation}\begin{split}(\ref{Sapprox})=&\left|\sum_{j\in \langle S\rangle}(e^{\rho(\dot\lambda_j+\imag\cdot\dot\beta_j)}-1)v_j\right|\leq \frac\delta{3dC_1C_2R}|v_{\langle S\rangle}|\\ \leq&\frac\delta{3C_1C_2R}\cdot C_1\leq\frac\delta3;\end{split}\end{equation}
and \begin{equation}\begin{split}(\ref{lambdaapprox})=&\left|\sum_{j\in I_{[\lambda]}}\big(e^{\rho(\dot\lambda_j+\imag\cdot\dot\beta_j)}-1-\rho(\dot\lambda_j+\imag\cdot\dot\beta_j)\big)v_j\right|\\
\leq&(\frac\delta {6dC_1C_2R})^2\cdot|v|\leq(\frac\delta {6dC_1C_2R})^2\cdot \frac{3d^2C_1^2C_2^2R^2}{\delta}\\
=&\frac\delta{12}.\end{split}\end{equation}

This completes the proof (ii), as well as that of the lemma.\end{proof}

So eventually Proposition \ref{line} is established by the argument following Lemma \ref{netlongline}.

\subsection{Density of lines in the torus}

Proposition \ref{line} reduces Proposition \ref{nonisolated} to the following claim:

\begin{proposition}\label{linerotation} Let $L$ be a line through the origin inside  $\bR^{r_1}\oplus\bC^{r_2}$. Then $\forall\epsilon,\delta>0$, $\exists\bfn\in H_{\epsilon,S}$ such that $\pi(\zeta^\bfn.L)$ is $\delta$-dense in $X$, where $\pi:\bR^{r_1}\oplus\bC^{r_2}\mapsto X$ is the natural projection.\end{proposition}

By ``$\delta$-dense'' we mean $\pi(\zeta^\bfn.L)$ is a $\delta$-net of the ambient space $X$.

In order to prove Proposition \ref{linerotation}. We are going to make use of the total irreducibility assumption from Theorem \ref{denseorcentral}.

Recall $d=r_1+2r_2$ and $\sigma_1\cdots,\sigma_{r_1}$ are real embeddings of $K$ while $\sigma_{r_1+1},\cdots,\sigma_d$ are the complex ones where $\sigma_{r_1+r_2+j}=\overline{\sigma_{r_1+j}}$.

\begin{lemma}\label{nonequivunits}Suppose $\{\alpha^\bfn:\bfn\in G_{\epsilon,S}\}$ contains a totally irreducible toral automorphism for all $\epsilon>0$ as assumed in Theorem \ref{denseorcentral}. Let $i_0,i_1,\cdots,i_m$ be distinct elements from $\{1,\cdots,d\}$. Suppose $m\geq 1$ and let $\Psi:\bZ^r\mapsto(\bC^\times)^m$ be the group morphism \begin{equation}\label{unitseqvclass}\Psi(\bfn)=\left(\frac{\zeta_{i_1}^\bfn}{\zeta_{i_0}^\bfn},\cdots,\frac{\zeta_{i_m}^\bfn}{\zeta_{i_0}^\bfn}\right).\end{equation}Then the image $\Psi(H_{\epsilon,S})$ has infinite size for all $\epsilon>0$.\end{lemma}

\begin{proof}Suppose for contradiction that $\Psi(H_{\epsilon,S})$ has finite size $N$ for some $\epsilon$. Let $q$ be the index of $H$ in $G=\bZ^r$. By assumption, there exists $\bfn\in G_{\frac\epsilon{Nq},S}$ such that $\alpha^{k\bfn}$ is irreducible for all $k\neq 0$.

By pigeonhole principle, there are two distinct $k,k'\in\{0,1,\cdots,q\}$ such that $k\bfn$, $k'\bfn$ belong to the same coset of $H$. Then $(k-k')\bfn$ is in both $H$ and $G_{\frac\epsilon N,S}$, hence in $H_{\frac\epsilon N,S}$.

Then $h(k-k')\bfn\in H_{\epsilon,S},\ \forall h\in\{0,1,\cdots,N\}$. As $|\Psi(H_{\epsilon,S})|=N$, there are $h\neq h'$ in $\{0,1,\cdots,N\}$ such that $\Psi\big(j(k-k')\bfn\big)=\Psi\big(h'(k-k')\bfn\big)$. So $(h-h')(k-k')\bfn\in H$ and $\Psi\big((h-h')(k-k')\bfn\big)=(1,1,\cdots,1)$, which in particular implies \begin{equation}\label{embcollapse}\zeta_{i_1}^{(h-h')(k-k')\bfn}=\zeta_{i_0}^{(h-h')(k-k')\bfn}.\end{equation} Recall $\zeta_{i_1}^{(h-h')(k-k')\bfn}$ and $\zeta_{i_0}^{(h-h')(k-k')\bfn}$ are respectively algebraic conjugates of $\zeta^{(h-h')(k-k')\bfn}$ by $\sigma_{i_1}$ and $\sigma_{i_0}$, two different embeddings of the number field $K$. (\ref{embcollapse}) actually shows $\zeta^{(h-h')(k-k')\bfn}$ belongs to some proper subfield of $K$. Thus by Remark \ref{alphazetairr}, $\alpha^\bfn$ is not a totally irreducible toral automorphism on $\bT^d$ because $\deg f<d$ and $(h-h')(k-k')\neq 0$, which contradicts our assumption. The proof is completed.\end{proof}

The dual group $\hat X$ of $X=\bR^{r_1}\oplus\bC^{r_2}/\Gamma$ consists of all real linear functionals $\xi:\bR^{r_1}\oplus\bC^{r_2}\mapsto\bR$ such that $\xi(\Gamma)\subset\bZ$ and the character $\xi:X\mapsto\bR/\bZ$ is defined by $\xi\big(\pi(v)\big)=\big(\xi(v) \mod \bZ\big), \forall v\in \bR^{r_1}\oplus\bC^{r_2}$.

As a linear functional, $\xi$ may be expressed as \begin{equation}\label{fourier}\xi(x)=\sum_{i=1}^{r_1}\xi_ix_i+\sum_{i=r_1}^{r_1+r_2}\big(\xi_i x_i+\xi_{i+r_2}\overline{x_i}\big),\end{equation} where $x=(x_1,\cdots,x_{r_1+r_2})$ with $x_1,\cdots,x_{r_1}\in\bR$, $x_{r_1+1},\cdots,x_{r_1+r_2}\in\bC$. In order that $\xi$ takes real values, $\xi_1,\cdots,\xi_{r_1}\in\bR$ and $\xi_{i+r_2}=\overline{\xi_i}$ for $i=r_1+1,\cdots,r_1+r_2$.

\begin{lemma}\label{eigenirrational}If $\xi\neq 0$ then $\xi_i\neq 0$, $\forall i=1,\cdots,d$.\end{lemma}
\begin{proof}Suppose $\xi_i=0$ for some $i$. If $i\leq r_1$ then $V_i\cong\bR$ and by (\ref{fourier}), $\xi|_{V_i}=0$. Suppose $i>r_1$, then $\overline{\xi_i}=0$ as well, and therefore, since $\xi_{i+r_2}=\overline{\xi_i}$, we may assume $r_1+1\leq i\leq r_1+r_2$. So in this case once again we have $\xi|_{V_i}=0$ by (\ref{fourier}). Because $\xi$ is non-trivial, $\xi^\bot=\{x\in X: \xi(x)=0\}$ is a proper closed subgroup of $X$. Since $\xi|_{V_i}=0$, $\xi^\bot$ contains $\pi(V_i)$; so  $\overline{\pi(V_i)}$ is a proper closed subgroup in $X$ as well. However, for $\bfn\in\bZ^r$, the multiplicative action by $\zeta^\bfn$ preserves $V_i\subset \bR^{r_1}\oplus\bC^{r_2}$, hence also preserves $\pi(V_i)$ and $\overline{\pi(V_i)}$. Therefore $\overline{\pi(V_i)}\subset X$ is a proper connected closed subgroup invariant under the $\bZ^r$-action $\zeta^\bfn\curvearrowright X$, which is conjugate to the $\bZ^r$-action $\alpha$ on $\bT^d$. Thus $\alpha$ admits a proper connected invariant closed subgroup in $\bT^d$, which is necessarily a subtorus. This violates the assumption that $\alpha$ contains irreducible toral automorphisms. Hence $\xi_i\neq 0$, $\forall i=1,\cdots,d$.\end{proof}

We borrow our next lemma from Berend's original proof.

\begin{lemma}\label{filling}\cite{B83}*{Lemma 4.7} Let $X$ be a compact abelian metric group  and $\hat X$ its Pontryagin dual. Suppose $\{X_k\}_{k=1}^\infty$ is a sequence of proper closed subgroups in $X$ satisfying: $\forall\xi\in\hat X\backslash\{0\}$, $X_k\notin\xi^\bot$ for sufficiently large $k$. Then for any $\delta>0$, $X_k$ is $\delta$-dense for sufficiently large $k$.\end{lemma}

Now we are ready to show Proposition \ref{linerotation}.

\begin{proof}[Proof of Proposition \ref{linerotation}] Denote the line $L$ by $\bR v$ for some vector $v=(x_i)_{i\in I}$ where $x_i\in\bR$ or $\bC$ according to whether $i\leq r_1$ or not. Denote $x_{r_1+r_2+j}=\overline{x_{r_1+j}}$ for $1\leq j\leq r_2$ and \begin{equation}I_L=\{i: 1\leq i\leq d, \ x_i\neq 0\},\end{equation} then $I_L$ is not empty. Write $I_L=\{i_0,\cdots,i_m\}$. Then $\forall\xi\in\hat X\backslash\{0\}$, $\xi(v)=\sum_{h=0}^m\xi_{i_h}x_{i_h}$ by (\ref{fourier}).\newline

{\noindent\bf Case 1}. If $m=0$ then take $\bfn=0\in H_{\epsilon,S}$. Note $\overline{\pi(L)}\subset X$  is a closed subgroup of $X$. For all $\xi\in\hat X\backslash\{0\}$, we have $\xi(v)=\xi_{i_0}x_{i_0}\neq 0$ thanks to Lemma \ref{eigenirrational}, so $\xi(\pi(av))=(a\xi(v)\mod\bZ)\neq 0$ for almost every $a\in\bR$. Hence as $\pi(av)\in \overline{\pi(L)}$, $\overline{\pi(L)}\notin\xi^\bot$ for all non-trivial $\xi$. Therefore $\overline{\pi(L)}=X$.\newline

{\noindent\bf Case 2}. If $m\geq 1$ then by applying Lemma \ref{nonequivunits} we obtain a sequence $\{\bfn_k\}_{k=1}^\infty\subset H_{\epsilon,S}$ such that the $\Psi(\bfn_k)$'s are all distinct where $\Psi$ is the group morphism in (\ref{unitseqvclass}).

Now fix $\xi\in\hat X\backslash\{0\}$ and consider $\xi(\zeta^{\bfn_k}.v)=\sum_{h=0}^m\xi_{i_h}\zeta_{i_h}^{\bfn_k} x_{i_h}$. Here $\xi_{i_h}x_{i_h}\neq 0,\forall h$ by Lemma \ref{eigenirrational} and the choices of $i_0,\cdots,i_m$. If $\xi(\zeta^{\bfn_k}.v)=0$ then \begin{equation}\label{uniteq}\xi_{i_0}x_{i_0}+\sum_{h=1}^m\xi_{i_h}x_{i_h}\frac{\zeta_{i_h}^{\bfn_k}}{\zeta_{i_0}^{\bfn_k}}=0.\end{equation}

(\ref{uniteq}) is a linear equation in the variable $\Psi(\bfn_k)=\left(\dfrac{\zeta_{i_h}^{\bfn_k}}{\zeta_{i_0}^{\bfn_k}}\right)_{h=1}^m\in(\bC^\times)^m$. Since $\Psi$ is a group morphism $\Psi(\bZ^r)$ is a subgroup of finite rank in $(\bC^\times)^m$. It follows from the estimate \cite{ESS02}*{Theorem 1.1} by Evertse, Schlikewei and Schmidt of number of solutions to unit equations in a multiplicative group , that there are only finitely many solutions to (\ref{uniteq}) in $\Psi(\bZ^r)$.

Because all the $\Psi(\bfn_k)$'s are different, when $k$ is sufficiently large (\ref{uniteq}) fails, i.e. $\xi(\zeta^{\bfn_k}.v)\neq0$.

Therefore  when $k$ is large enough, for almost all  $a\in\bR$ the expression $\xi\big(\pi(a\zeta^{\bfn_k}.v)\big)=\big(a\xi(\zeta^{\bfn_k}.v)\text{ mod }\bR/\bZ\big)$ doesn't vanish. As $a\pi(\zeta^{\bfn_k}.v)\in\pi\big(\zeta^{\bfn_k}.L\big)$. The closed subgroup $\overline{\pi\big(\zeta^{\bfn_k}.L\big)}\subset X$ is not contained in $\xi^\bot$. Since $\xi$ is an arbitrary non-trivial character, Lemma \ref{filling} claims $\forall\delta>0$, $\overline{\pi\big(\zeta^{\bfn_k}.L\big)}$ is $\delta$-dense in $X$ when $k$ is large enough, and thus so is $\pi\big(\zeta^{\bfn_k}.L\big)$ itself. Proposition \ref{linerotation} is proved.\end{proof}

Finally we are ready to give the proof of Proposition \ref{nonisolated}.

\begin{proof}[Proof of Proposition \ref{nonisolated}] If $A_{\epsilon_0}$ contains a $V_{\langle S\rangle}$ pattern for some $\epsilon_0<\epsilon$, then by Proposition \ref{line} $A_{\frac{\epsilon_0+\epsilon}2}$ contains $y+L$  where $y\in X$, and $L$ is a line in $\bR^{r_1}\oplus\bC^{r_2}$. Take $\bfn\in H_{\frac{\epsilon-\epsilon_0}2,S}$ such that $\pi(\zeta^\bfn.L)$ is $\delta$-dense in $X$. Then $\zeta^\bfn.(y+L)$, which is the translate of $\pi(\zeta^\bfn.L)$ by $\zeta^\bfn.y\in X$, is $\delta$-dense as well; and therefore so is $A_\epsilon$ because $A_\epsilon\supset H_{\frac{\epsilon-\epsilon_0}2,S}. A_{\frac{\epsilon_0+\epsilon}2}\supset\zeta^\bfn.(y+L)$. $A_\epsilon$ is actually dense in $X$ because $\delta$ can be arbitrarily small. So $A_\epsilon=X$ by closedness.\end{proof}

\section{Proof of the main results}

\subsection{Characterization of minimal $(G,S)$-invariant sets} In order to be able to use what was showed in the previous section, we hope to show every $(G,S)$-invariant set contains a $V_{\langle S\rangle}$-translated torsion point.

We need two lemmas to establish this fact. The first one is quite simple and the second one relies on Proposition \ref{minunion}.

\begin{lemma}\label{recurtorsion}Suppose $\bfm,\bfn\in\bZ^r$ are not equal,

(i) If $x\in X$ satisfies $\zeta^\bfm.x=\zeta^\bfn.x$, then $x$ is a torsion point;

(ii) If $\zeta^\bfm.x=\zeta^\bfn.x+v$ where $v\in V_{\langle S\rangle}$, then $x$ is a $V_{\langle S\rangle}$-translated torsion point; \end{lemma}

\begin{proof}(i) Write $x=\pi(\tilde x)$ where $\tilde x\in\bR^{r_1}\oplus\bC^{r_2}$. Then $\zeta^\bfm.\tilde x-\zeta^\bfn.\tilde x\in\Gamma$. Recall $\bfm\mapsto\zeta^\bfm$ is a group embedding of $\bZ^r$ into the group of units $U_K$. So if $\bfm\neq \bfn$ then $\zeta^\bfm\neq\zeta^\bfn$ and we have $(\zeta^\bfm-\zeta^\bfn)^{-1}\in K$. So $\tilde x\in (\zeta^\bfm-\zeta^\bfn)^{-1}.\Gamma$ where the multiplication is given by (\ref{fieldmulti}). Recall $\bR^{r_1}\oplus\bC^{r_2}\cong K\otimes_\bQ\bR$ via the map $\sigma$ given in (\ref{fieldemb}), and $\Gamma\subset\sigma(K)$ is a full-rank sublattice of $\bR^{r_1}\oplus\bC^{r_2}$. Thus $\tilde x\in (\zeta^\bfm-\zeta^\bfn)^{-1}.\sigma(K)=\sigma\big((\zeta^\bfm-\zeta^\bfn)^{-1}.K\big)=\sigma(K)$ where we used (\ref{embmulti}). Because $\Gamma$ has full rank, its $\bQ$-span is $\sigma(K)$ and contains $\tilde x$. Therefore there is an integer $q$ such that $q\tilde x\in\Gamma$. So $qx=0$ in $X=(\bR^{r_1}\oplus\bC^{r_2})/\Gamma$.

(ii) Let $x'=x+(\zeta^\bfm-\zeta^\bfn)^{-1}.v\in X$ then $\zeta^\bfm.x'=\zeta^\bfn.x'$. By (i), $x'$ is of torsion. Moreover $(\zeta^\bfm-\zeta^\bfn)^{-1}.v\in V_{\langle S\rangle}$. Thus $x$ is a $V_{\langle S\rangle}$-translated torsion point.\end{proof}

\begin{lemma}\label{minfull} There doesn't exist a pair of minimal $(G,S)$-invariant closed sets $M$, $M'$ such that $M+M'=X$. \end{lemma}
\begin{proof}Assume there is such a pair of sets. By Lemma \ref{mingene} we may write $M=\bigcap_{\epsilon>0}\overline{G_{\epsilon,S}.x}\subset X$ and $M'=\bigcap_{\epsilon>0}\overline{G_{\epsilon,S}.x'}\subset X$.

Fix a prime number $p$. Define a sequence of congruence subgroups \begin{equation}G^{(f)}=\{\bfn\in\bZ^r:\zeta^\bfn\gamma\equiv\gamma(\text{mod }p^f\Gamma),\forall \gamma\in \Gamma\}\end{equation} for all $f\geq 0$. Then $G=G^{(0)}>G^{(1)}>G^{(2)}>\cdots$.

Furthermore, $G^{(f)}$ is of finite index in $G$ for all $\forall f\geq 0$.  To see this, observe the multiplication by any $\zeta^\bfn$ preserves $p^f\Gamma$ since it preserves $\Gamma$. Thus the multiplicative action $\zeta$ induces a $\bZ^r$-action on the finite abelian group $\Gamma/p^f\Gamma$, in other words, a group morphism from $G=\bZ^r$ to $\Aut(\Gamma/p^f\Gamma)$. $G^{(f)}$ is just the kernel of this morphism, which has finite index as $|\Aut(\Gamma/p^f\Gamma)|<\infty$.\newline

{\bf Claim. \it There exists a decreasing sequence of non-empty closed sets $M=M^{(0)}\supset M^{(1)}\supset M^{(2)}\supset\cdots$, such that $\forall f\geq 0$, $M^{(f)}$ is a minimal $(G^{(f)},S)$-invariant closed set and $M^{(f)}+M'=X$.}\newline

We prove the claim by induction. When $f=0$ the claim is part of the condition in the lemma. Suppose there is already a finite sequence $M=M^{(0)}\supset M^{(1)} \supset\cdots\supset M^{(f)}$ satisfying the claim, we want to construct $M^{(f+1)}\subset M^{(f)}$.

As $G^{(f+1)}$ is of finite index in $G^{(f)}$, Proposition \ref{minunion} applies and produces a finite collection of minimal $(G^{(f+1)},S)$-invariant closed sets $M^{(f+1)}_1,\cdots,M^{(f+1)}_N\subset M^{(f)}$, whose union is $M^{(f)}$. Applying Proposition \ref{minunion} to $G$ and $G^{(f+1)}$ we can also find another finite collection of minimal $(G^{(f+1)},S)$-invariant closed sets
$(M')^{(f+1)}_1,\cdots,(M')^{(f+1)}_{N'}\subset M'$ whose union is $M'$. It is already known that $M^{(f)}+M'=X$. Hence $\bigcup_{\substack{n=1,\cdots,N\\n'=1,\cdots,N'}}(M^{(f+1)}_n+(M')^{(f+1)}_{n'})=X$. In consequence, because all the $(M^{(f+1)}_n+(M')^{(f+1)}_{n'})$'s are closed there exists a pair $(n,n')$ such that $M^{(f+1)}_n+(M')^{(f+1)}_{n'}$ has non-empty interior. In particular, $M^{(f+1)}_n+(M')^{(f+1)}_{n'}$ contains a torsion point and an open neighborhood of it.

Therefore $M^{(f+1)}_n+(M')^{(f+1)}_{n'}$ contains a $V_{\langle S\rangle}$-pattern. By Lemma \ref{mingene}, $M^{(f+1)}_n=\bigcap_{\epsilon>0}\overline{(G^{(f+1)})_{\epsilon,S}.z}$ and $(M')^{(f+1)}_{n'}=\bigcap_{\epsilon>0}\overline{(G^{(f+1)})_{\epsilon,S}.z'}$ for some $z,z'\in X$. Therefore $A=M^{(f+1)}_n+(M')^{(f+1)}_{n'}$ verifies condition (\ref{fgnonisolated}) for $H=G^{(f+1)}$. By Proposition \ref{nonisolated}, $M^{(f+1)}_n+(M')^{(f+1)}_{n'}=X$. In particular, $M^{(f+1)}_n+M'=X$ as $(M')^{(f+1)}_{n'}\subset M'$. The inductive step is proved by setting $M^{(f+1)}=M^{(f+1)}_n$, this verifies the claim.\\

We fix a point $x_\infty$ in the limit set $\bigcap_{f=0}^\infty M^{(f)}$, which is non-empty by compactness of $X$.

For any given $f\in\bN$ take a point $y\in X$ of the form \begin{equation}\label{ppowertorsion}y=\pi(\tilde y),\text{ where }\tilde y\in p^{-f}\Gamma\subset\bR^{r_1}\oplus\bC^{r_2}\end{equation} Then $p^fy=0$ in $X$ and $y$ is a torsion point. By the claim, $y=u+u'$ where $u\in M^{(f)}$, $u'\in M'$. Since $M^{(f)}$ is minimal and contains $x_\infty$, $x_\infty\in\bigcap_{\epsilon>0}\overline{(G^{(f)})_{\epsilon,S}.u}$ by Lemma \ref{mingene}. According to Remark \ref{describeinv}, there is a sequence $\{\bfn_k\}_{k=1}^\infty$ such that $\bfn_k\in G^{(f)}_{\epsilon_k,S}$, where $\lim_{k\rightarrow\infty}\epsilon_k=0$, and \begin{equation}\label{minfull1}\lim_{k\rightarrow\infty}\zeta^{\bfn_k}.u=x_\infty.\end{equation}

Suppose $\tilde y=p^{-f}\gamma$ with $\gamma\in\Gamma$, then $\zeta^{\bfn_k}.\tilde y=p^{-f}\zeta^{\bfn}\gamma$. Since $\bfn\in G^{(f)}$, $\zeta^{\bfn}.\gamma\equiv\gamma(\text{mod } p^f\Gamma)$, therefore $\zeta^{\bfn_k}.\tilde y\equiv p^{-f}\gamma=\tilde y (\text{mod }\Gamma)$ and the projection $y=\pi(\tilde y)$ satisfies \begin{equation}\label{minfull2}\zeta^{\bfn_k}.y=y.\end{equation} Take the difference between (\ref{minfull1}) and (\ref{minfull2}), we obtain:
\begin{equation}\lim_{k\rightarrow\infty}\zeta^{\bfn_k}.u'=y-x_\infty.\end{equation}

As $\bfn_k\in G^{(f)}_{\epsilon_k,S}$ with $\epsilon_k\rightarrow0$, it follows from Remark \ref{describeinv} that $y-x_\infty$ belongs to $\bigcap_{\epsilon>0}\overline{(G^{(f)})_{\epsilon,S}.u'}$. As $M'$ is $(G,S)$-invariant, it is also $(G^{(f)},S)$-invariant; so $\bigcap_{\epsilon>0}\overline{(G^{(f)})_{\epsilon,S}.u'}\subset M'$ by definition. Thus $y-x_\infty\in M'$ for all points $y$ of the form (\ref{ppowertorsion}). However when $f$ tends to $\infty$, because $p^{-f}\Gamma$ becomes dense in $\bR^{r_1}\oplus\bC^{r_2}$, points $y$ of the form (\ref{ppowertorsion}) become dense in $X$, and so do the $(y-x_\infty)$'s since $x_\infty$ is fixed. Hence $M'$ is dense and as a closed set it must be $X$. This contradicts the minimality of $M'$ as a $(G,S)$-invariant closed set because $X$ contains proper $(G,S)$-invariant closed subsets. The proof is completed. \end{proof}

\begin{proposition}\label{minclassify}Suppose $M\subset X$ is a minimal $(G,S)$-invariant closed set. Then $M$ contains a $V_{\langle S\rangle}$-translated torsion point.\end{proposition}

Remark that together with Lemma \ref{mingene}, the proposition implies $M=\bigcap_{\epsilon>0}\overline{G_{\epsilon,S}.x}$ where $x$ is a $V_{\langle S\rangle}$-translated torsion point. As a result, $M$ is finite by Lemma \ref{torsionorbit}.

\begin{proof}[Proof of Proposition \ref{minclassify}] By Lemma \ref{mingene}, $M=\bigcap_{\epsilon>0}\overline{G_{\epsilon,S}.x}$ for some $x\in X$. We distinguish between three cases.\\

{\noindent\bf Case 1}. If $\exists\epsilon>0$ such that $G_{\epsilon,S}.x$ is finite then $x$ is a torsion point. Actually, because it follows from Lemma \ref{nongroupapprox} that $G_{\epsilon,S}$ is infinite, $\zeta^\bfm.x=\zeta^\bfn.x$ for some pair $\bfm\neq\bfn$ from $G_{\epsilon,S}$. By Lemma \ref{recurtorsion}, $x$ is of torsion so we are done. \\

{\noindent\bf Case 2}. Suppose $\exists\epsilon>0$ such that $G_{\epsilon,S}.x$ is infinite but for any coverging sequence $\{y_k\}_{k=1}^\infty\subset\overline{G_{\epsilon,S}.x}$, whose limit we denote by $y$, $y_k-y\in V_{\langle S\rangle}$ for sufficiently large $k$. (Here as in Definition \ref{pattern},  $y_k-y$ is regarded as a very short vector in $\bR^{r_1}\oplus\bC^{r_2}$ when $y_k$ is sufficiently close to $y$ and $y_k-y\rightarrow 0$ as $k\rightarrow\infty$.) In this case we claim $x$ is a $V_{\langle S\rangle}$-translated torsion point.

In fact since $G_{\epsilon,S}.x$ is infinite and $X$ is compact, there is always such a coverging sequence $\{y_k\}_{k=1}^\infty$ inside $G_{\epsilon,S}.x$ where the $y_k$'s are all distinct. By assumption for very large $k$, $y_k-y\in V_{\langle S\rangle}$ where $y=\lim_{k\rightarrow\infty}y_k\in \overline{G_{\epsilon,S}.x}$. Thus for $k$, $k'$ both sufficiently large , $y_k=y_{k'}+v$ for some $v\in V_{\langle S\rangle}$. Since $y_k=\zeta^{\bfn_k}.x$ and $y_k=\zeta^{\bfn_{k'}}.x$ for some $\bfn_k,\bfn_{k'}\in G_{\epsilon,S}$, $x$ is a $V_{\langle S\rangle}$-translated torsion point by Lemma \ref{recurtorsion}.\\

{\noindent\bf Case 3}. If neither Case 1 nor Case 2 holds, then for all given $\epsilon>0$, $G_{\epsilon,S}.x$ is infinite and  moreover there is a sequence $\{y_k\}_{k=1}^\infty\subset\overline{G_{\epsilon,S}.x}$ coverging to some $y\in \overline{G_{\epsilon,S}.x}$ such that $y_k-y\notin V_{\langle S\rangle}, \forall k$.

The opposite operator $x\mapsto -x$ on $X$ commutes with the $\bZ^r$-action. Therefore as $M$ is a minimal $(G,S)$-invariant closed set, so is $-M=\{-y:y\in M\}=\bigcap_{\epsilon>0}\overline{G_{\epsilon,S}.(-x)}$.
By Lemmas \ref{invpm} and \ref{fginvariant} $A=M-M$ is $(G,S)$-invariant and is equal to $\bigcap_{\epsilon>0}\left(\overline{G_{\epsilon,S}.x}-\overline{G_{\epsilon,S}.x}\right)$.

Let $A_\epsilon=\overline{G_{\epsilon,S}.x}-\overline{G_{\epsilon,S}.x}$. Remark for  $\{y_k\}_{k=1}^\infty$ and $y$  above, $\pi(y_k-y)\in A_\epsilon$ for all $k$. The earlier characterization of $\{y_k\}_{k=1}^\infty$ and $y$ actually says $A_\epsilon$ contains a $V_{\langle S\rangle}$-pattern at $0$. It follows from Proposition \ref{nonisolated} that $A_\epsilon=X$ for all $\epsilon>0$ and therefore $M-M=\bigcap_{\epsilon>0}A_\epsilon=X$; which contradicts Lemma \ref{minfull}. So Case 3 cannot happen. This concludes the proof of Proposition \ref{minclassify}.\end{proof}

\subsection{Proof of the main theorems}

\begin{proof}[Proof of Theorem \ref{denseorcentral'}]Theorem \ref{denseorcentral'}.(i) is already covered by Lemma \ref{torsionorbit}, so we only need to prove part (ii), i.e. for any $x\in X$ which is not a $V_{\langle S\rangle}$-translated torsion point, the set $A_\epsilon=\overline{G_{\epsilon,S}.x}$ is equal to $X$ for all $\epsilon$.

Let $A=\bigcap_{\epsilon>0}\overline{G_{\epsilon,S}.x}$, which is $(G,S)$-invariant by Lemma \ref{cosetcomposlice}. It contains a minimal $(G,S)$-invariant set by Lemma \ref{minZorn}, thus contains a $V_{\langle S\rangle}$-translated torsion point $y$ by Proposition \ref{minclassify}.

Fix an arbitrary $\epsilon>0$, we claim that $A_{\frac\epsilon2}$ contains a $V_{\langle S\rangle}$-pattern near the given point $y$; i.e. a sequence $\{y_k\}_{k=1}^\infty$ converging to $y$, such that $y_k-y\notin V_{\langle S\rangle}$.  In fact, suppose not, then there are only two possibilities:\\

{\noindent\bf Case 1.} $y\in G_{\frac\epsilon2,S}.x$. Hence $x\in G_{\frac\epsilon2,S}.y$ because it is clear that $\bfn\in G_{\frac\epsilon2,S}$ if and only if $-\bfn\in G_{\frac\epsilon2,S}$.  Therefore by Lemma \ref{torsionorbit}, $x$ is a $V_{\langle S\rangle}$-translated torsion point itself; contradiction.\\

{\noindent\bf Case 2.} $y\notin G_{\frac\epsilon2,S}.x$, in other words $y$ is an accumulation point of $G_{\frac\epsilon2,S}.x$. Since $V_{\langle S\rangle}$-pattern is not allowed by assumption, there must be a convergent sequence $\{y_k\}_{k=1}^\infty\subset G_{\frac\epsilon2,S}.x$ whose limit is $y$, such that $y_k$ is in the $V_{\langle S\rangle}$-foliation through $y$ for sufficiently large $k$ (or equivalently, for all $k$ by choosing a subsequence). Hence the point $y_k$ itself is a $V_{\langle S\rangle}$-translated torsion point because $y$ is. But $y_k\in G_{\frac\epsilon2,S}.x$. Again by Lemma \ref{torsionorbit}, $x\in G_{\frac\epsilon2,S}.y_k$ is a $V_{\langle S\rangle}$-translated torsion point and this contradicts the assumption on $x$.

Therefore the claim that $A_{\frac\epsilon2}$ contains a $V_{\langle S\rangle}$-pattern is verified. But then it follows from Proposition \ref{nonisolated} that $A_\epsilon=X$. This establishes Theorem \ref{denseorcentral'}.\end{proof}

\begin{proof}[Proof of Theorem \ref{denseorcentral}] As via the transformation $\psi$, the actions $\alpha:\bZ^r\curvearrowright\bT^d$ and $\zeta:\bZ^r\curvearrowright X$ are equivalent, it suffices to prove the part concerning $X$ in Theorem \ref{denseorcentral} and the analogous claim about $\bT^d$ would follow.

If $x\in X$ is a $V_{\langle S\rangle}$-translated torsion point, then by Lemma \ref{central}, the set (\ref{Xpartialorb}) is not dense.

On the other hand, if $x$ is not a $V_{\langle S\rangle}$-translated torsion point, then by Theorem \ref{denseorcentral'}, $\overline{G_{\epsilon,S}.x}=X$. Because the set (\ref{Xpartialorb}) contains $G_{\epsilon,S}.x$ it is also dense in $X$, which completes the proof.\end{proof}

{\noindent\bf Acknowledgments:} I want to thank my thesis advisor Prof. Elon Lindenstrauss for initiating the larger project of which this paper is a part, and for his constant support. Thanks also go to the referee for helpful comments. I am grateful to the Hebrew University of Jerusalem for its hospitality.

\end{document}